\documentclass[11pt]{article}

\usepackage{hyperref}

\usepackage{amsmath,amsfonts,amssymb,amsthm}
\usepackage{mathrsfs}
\usepackage{bbm}
\usepackage{bm}
\usepackage{mathtools}
\usepackage{graphicx}
\usepackage{color}
\usepackage{pgf,tikz}
\usetikzlibrary{arrows}
\definecolor{zzttqq}{rgb}{0.6,0.2,0.}

\newtheorem{theorem}{Theorem}[section]
\newtheorem{lemma}[theorem]{Lemma}

\newtheorem{corollary}[theorem]{Corollary}
\newtheorem{proposition}[theorem]{Proposition}

\theoremstyle{remark}

\newtheorem{rem}[theorem]{Remark}
\renewenvironment{proof}[1][Proof]{\noindent{\itshape {#1.} } }{$\Box$ 
\medskip} 

\numberwithin{equation}{section}
\newcommand{\R}{\mathbb{R}}
\newcommand{\Z}{\mathbb{Z}}

\newcommand{\Pb}{\mathbb{P}}
\newcommand{\E}{\mathbb{E}}
\newcommand{\F}{\mathcal{F}}
\newcommand{\N}{\mathbb{N}}

\newcommand{\eps}{\varepsilon}
\newcommand{\ueps}{u_\eps}
\newcommand{\uepsp}{u_\eps^+}
\newcommand{\uepsm}{u_\eps^-}
\newcommand{\uepspm}{u_\eps^\pm}
\newcommand{\uepsext}{u^{\rm{ext}}_\eps}

\newcommand{\Depsp}{D_\eps^+}
\newcommand{\Depsm}{D_\eps^-}
\newcommand{\Depspm}{D_\eps^\pm}
\newcommand{\Gameps}{\Gamma_\eps}
\newcommand{\Gamd}{\mathbf{\Gamma}_d}
\newcommand{\Rdp}{\mathbb{R}_d^+}
\newcommand{\Rdm}{\mathbb{R}_d^-}
\newcommand{\Rdpm}{\mathbb{R}_d^{\pm}}

\newcommand{\Atilde}{\widetilde{A}}

\newcommand{\Acal}{\mathcal{A}}
\newcommand{\Ieps}{\mathcal{I}_\eps}

\newcommand{\Hcal}{\mathcal{H}}
\newcommand{\Wcal}{\mathcal{W}}

\newcommand{\cA}{\mathcal{A}}

\newcommand{\dist}{\mathrm{dist}}
\title{Homogenization of Randomly Deformed Conductivity Resistant Membranes}
\author{
Wenjia Jing\thanks{Department of Mathematics, University of Chicago, 5734 S. University Avenue, Chicago, IL 60637 (wjing@math.uchicago.edu)}}

\begin{document}
\maketitle

\begin{abstract}

We study the homogenization of a stationary conductivity problem in a random heterogeneous medium with highly oscillating conductivity coefficients and an ensemble of simply closed conductivity resistant membranes. This medium is randomly deformed and then rescaled from a periodic one with periodic membranes, in a manner similar to the random medium proposed by Blanc, Le Bris and Lions \cite{BLBL06}. Across the membranes, the flux is continuous but the potential field itself undergoes a jump of Robin type. We prove that, for almost all realizations of the random deformation, as the small scale of variations of the medium goes to zero, the random conductivity problem is well approximated by that on an effective medium which has deterministic and constant coefficients and contains no membrane. The effective coefficients are explicitly represented. One of our main contributions is to provide a solution to the associated auxiliary problem that is posed on the whole domain with infinitely many interfaces, in a setting that is neither periodic nor stationary ergodic in the usual sense. 


\end{abstract}

\section{Introduction}
\label{sec:intro}

In this article, we investigate the stochastic homogenization problem for a second order elliptic equation of divergence form that is posed on domains separated by an ensemble of simply closed surfaces, with jump type transmission conditions across them. The surfaces that separate the spatial domain have length scale $\eps \ll 1$ and they are realized as a random deformation from a periodic structure of surfaces. Our goal is to study the behavior, as $\eps$ goes to zero, of the solution to this equation.

More precisely, let $D$ be an open bounded subset in $\R^d$, $d = 2,3$, which is divided by a random ensemble of simply closed interfaces $\Gameps$ into $\Depsp$ and $\Depsm$, where $\Depsm$ denotes the union of the interiors enclosed by the interfaces in $\Gameps$, and $\Depsp$ denotes the rest of the domain. The small parameter $0<\eps \ll 1$ is the length scale of interfaces. We study the behavior of $u_\eps = \uepsp(x,\omega) \chi_{\Depsp} + \uepsm(x,\omega)\chi_{\Depsm}$, where $\chi_U$ denotes the indicator function of an open set $U$, and $\ueps$ solves the following problem:
\begin{equation}
\left\{
\begin{aligned}
&-\nabla \cdot\left( A^\eps (x,\omega) \nabla \uepspm(x,\omega)\right) = f(x), &\quad &\text{ for } x \in \Depspm,\\
&\frac{\partial}{\partial \nu_{A^\eps}} \uepsp(x,\omega) = \frac{\partial}{\partial \nu_{A^\eps}} \uepsm(x,\omega), &\quad &\text{ for } x \in \Gameps,\\
&\uepsp(x,\omega) - \uepsm(x,\omega) = \eps \frac{\partial}{\partial \nu_{A^\eps}} \uepsp(x,\omega), &\quad &\text{ for } x \in \Gameps,\\
&\uepsp(x,\omega) = 0, &\quad &\text{ for } x \in \partial D.
\end{aligned}
\right.
\label{eq:aniso}
\end{equation}
The first line in \eqref{eq:aniso} is understood as two equations for, respectively, $\uepsp$ on $\Depsp$ and $\uepsm$ on $\Depsm$. The variable $\omega$ denotes the realization of the random ensemble $\Gameps$, which is obtained by a random deformation of a periodic ensemble followed by rescaling. Notice that $\Gameps, \Depsp, \Depsm$ and hence $u_\eps$ are all random.

The problem above models, among many other natural applications, the stationary conductivity of heat through a medium that contains heat resistant membranes $\Gameps$. The anisotropic diffusion coefficients $A^\eps = (a^\eps_{ij})$ form a $d\times d$ matrix with entries 
\begin{equation*}
a_{ij}^\eps(x,\omega)= \widetilde{a}_{ij}\left(\Phi^{-1}\left(\frac{x}{\eps},\omega\right)\right),
\end{equation*}
where $\Phi(\cdot,\omega)$ is a random diffeomorphism on $\R^d$ and $\Atilde(y) = (\widetilde{a}_{ij}(y))$ is a $[0,1)^d$-periodic uniformly elliptic matrix. For simplicity, we assume that $(\widetilde{a}_{ij})$ is symmetric. The second and third equations in \eqref{eq:aniso} are the transmission conditions across the membranes. There, we define the conormal derivative of $\uepspm$, i.e. the normal flux, at $x \in \Gameps$ as
\begin{equation}
\frac{\partial \uepspm}{\partial \nu_{A^\eps}} : = \nu_x \cdot A^\eps \nabla \uepspm(x,\omega),
\label{eq:conormal}
\end{equation}
where $\nu_x$ is the unit outer normal vector along the boundary $\Gameps$ of $\Depsm$. The transmission conditions depict that the flux is continuous across the interface while the potential field $\ueps$ itself has a jump which is proportional to the flux. This jump condition is obtained asymptotically by considering the membranes as interfacial thermal barriers whose thickness is sent to zero. We refer to Carslaw and Jaeger \cite{CJ47} for the physical justification of these conditions, and the book of Milton \cite{miltonbook} for a comprehensive treatment of composite materials.

Equations with highly oscillating coefficients and/or highly oscillating domains arise naturally in many applications in physics and engineering. Due to the small scale variations, it is difficult to study such equations directly. For instance, straightforward numerical simulations of such equations become a daunting task when the scale is very small. It is hence plausible to seek for simplified equations which approximate the heterogeneous ones when the small scale tends to zero, under certain assumptions on the coefficients and the problem settings, e.g. periodicity or stationary ergodicity. This is the well known homogenization theory, which has a long history that dates back to the 70's; see e.g. Bensoussan, Lions and Papanicolaou \cite{BLP-78} and Tartar \cite{Tar_cour} for the periodic setting, and Papanicolaou and Varadhan \cite{PV79} and Kozlov \cite{Kozlov} for the random setting. We refer to the books of Zhikov, Kozlov and Ole\u{\i}nik \cite{Jikov_book} for a comprehensive treatment of homogenization theory. 

In this paper we study homogenization of \eqref{eq:aniso} where both the elliptic coefficients and the interfaces are random and vary on a scale of $\eps$. More precisely, our random setting is obtained by a random deformation from the corresponding periodic setting, in which the coefficients and the interfaces are periodic. Our idea takes inspiration from the random settings of Blanc, Le Bris and Lions \cite{BLBL06,BLBL07}. Details of the setting are in Section \ref{sec:form}. The resulting medium is $\Z^d$-stationary ergodic, which is different from the usual $\R^d$-stationary ergodic. 

Our main result, Theorem \ref{thm:homog} below, shows that as $\eps \to 0$, for almost all $\omega$, the unique solution $u_\eps(\cdot, \omega)$ of \eqref{eq:aniso} converges to the solution of the following deterministic equation
\begin{equation}
\left\{
\begin{aligned}
-\nabla \cdot  A^0\ \nabla u_0(x) = f(x), \quad &\text{ for } x \in D,\\
u_0(x) = 0, \quad &\text{ for } x \in \partial D.
\end{aligned}
\right.
\label{eq:homog}
\end{equation}
We note the effective medium does not contain any conductivity resistant membrane, $\ueps$ converges strongly in $L^2(D)$ to $u_0$, and the flux $\chi_{\Depsp} A^\eps \nabla \ueps^+ + \chi_{\Depsm} A^\eps\nabla \ueps^-$ converges weakly in $(L^2(D))^d$ to the homogenized flux $A^0\nabla u_0$. The homogenized elliptic coefficients $(A^0)_{ij}$ are deterministic constants given by the formula \eqref{eq:Abardef}.

The homogenization problem of \eqref{eq:aniso} in the periodic case was first studied by Monsurr\`o \cite{Monsur}. Later, the parabolic version was studied by Donato and Monsurr\`o \cite{DM04}, and the wave equation case was studied by Donato, Faella and Monsurr\`o \cite{DFM07}. Since the interfaces divide the physical domain of the equation, homogenization of equations with interfaces is closely related to homogenization in perforated domains. The study of the latter problem goes back at least to Cioranescu and Saint Jean Paulin \cite{CSP}, and a general framework for periodic perforations was developed by Cioranescu and Murat \cite{CM82}. The main tool in \cite{CSP} was the construction of an operator that extends functions to interior of the perforations, which was used also in \cite{Monsur,DM04,DFM07}. Allaire and Murat \cite{AM93} studied homogenization of Neumann problem on perforated domains without using this extension operator. Recently, Allaire and Habibi \cite{AH13,AH132} studied the interface problem using the two-scale convergence method. In the random setting, homogenization in perforated domain was studied by Zhikov \cite{Zhikov90}. To our best knowledge, random homogenization of the interface problem \eqref{eq:aniso} was first studied by the author with Ammari, Garnier, Giovangigli and Seo \cite{AGGJS13}, where randomly deformed structure with constant isotropic conductivity is considered and we constructed a random extension operator following the ideas of \cite{CSP,Monsur}. The current paper generalizes the result of \cite{AGGJS13} and deals with both random interfaces and random coefficients.

\begin{figure}[t]
\caption{\label{fig1}Left: an interface $\Gamma$ divides the unit cell $Y = [0,1]^d$ into $Y^-$ and $Y^+$; Middle: the structure $\Gamd$ over $\R^d$ formed by translating $\Gamma$ periodically; Right: a typical realization of the deformed structure, i.e. $\Phi(\Gamd, \omega)$.}
\begin{center}
\includegraphics[width=0.45\textwidth]{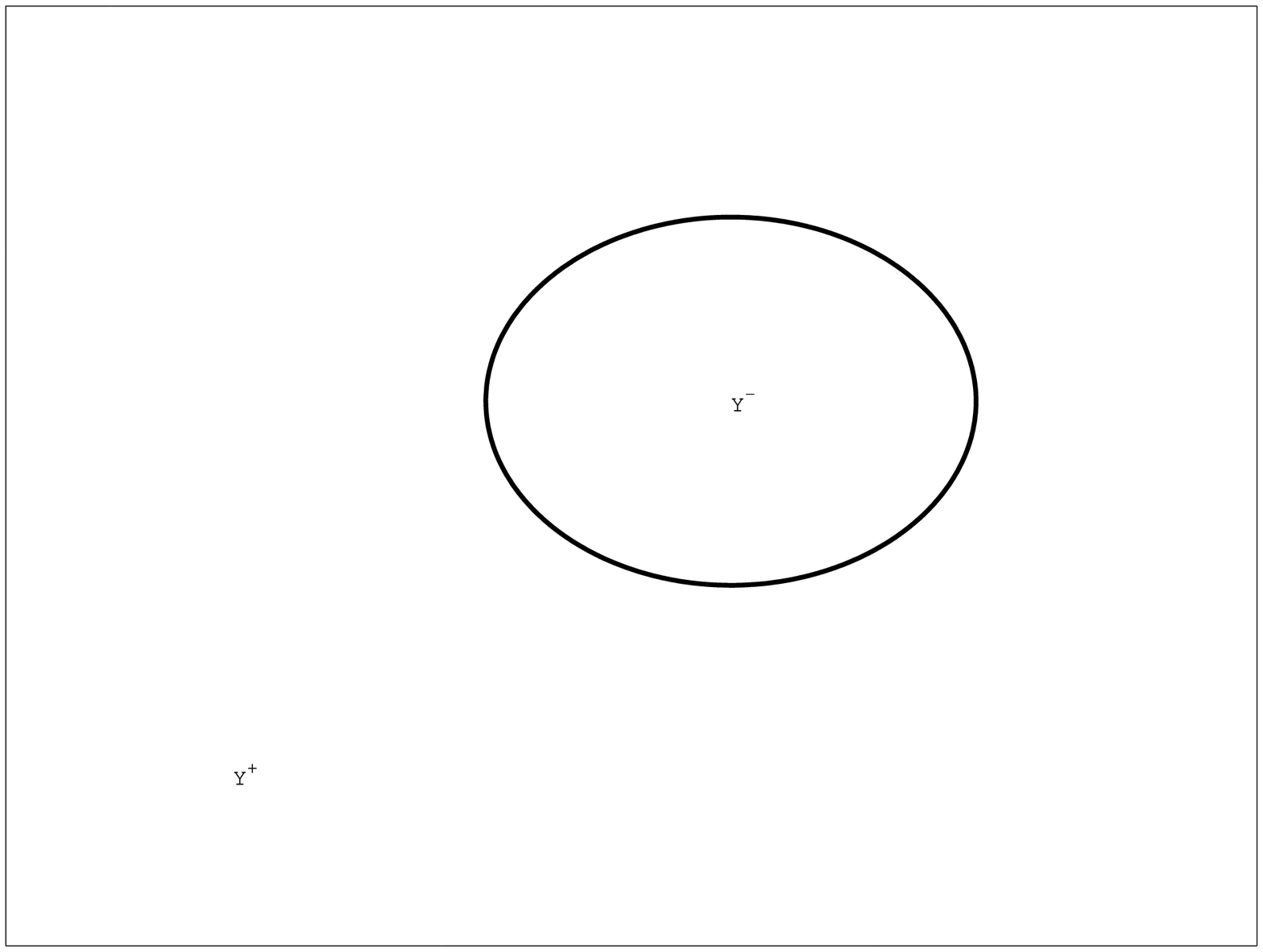}
\includegraphics[width=0.45\textwidth]{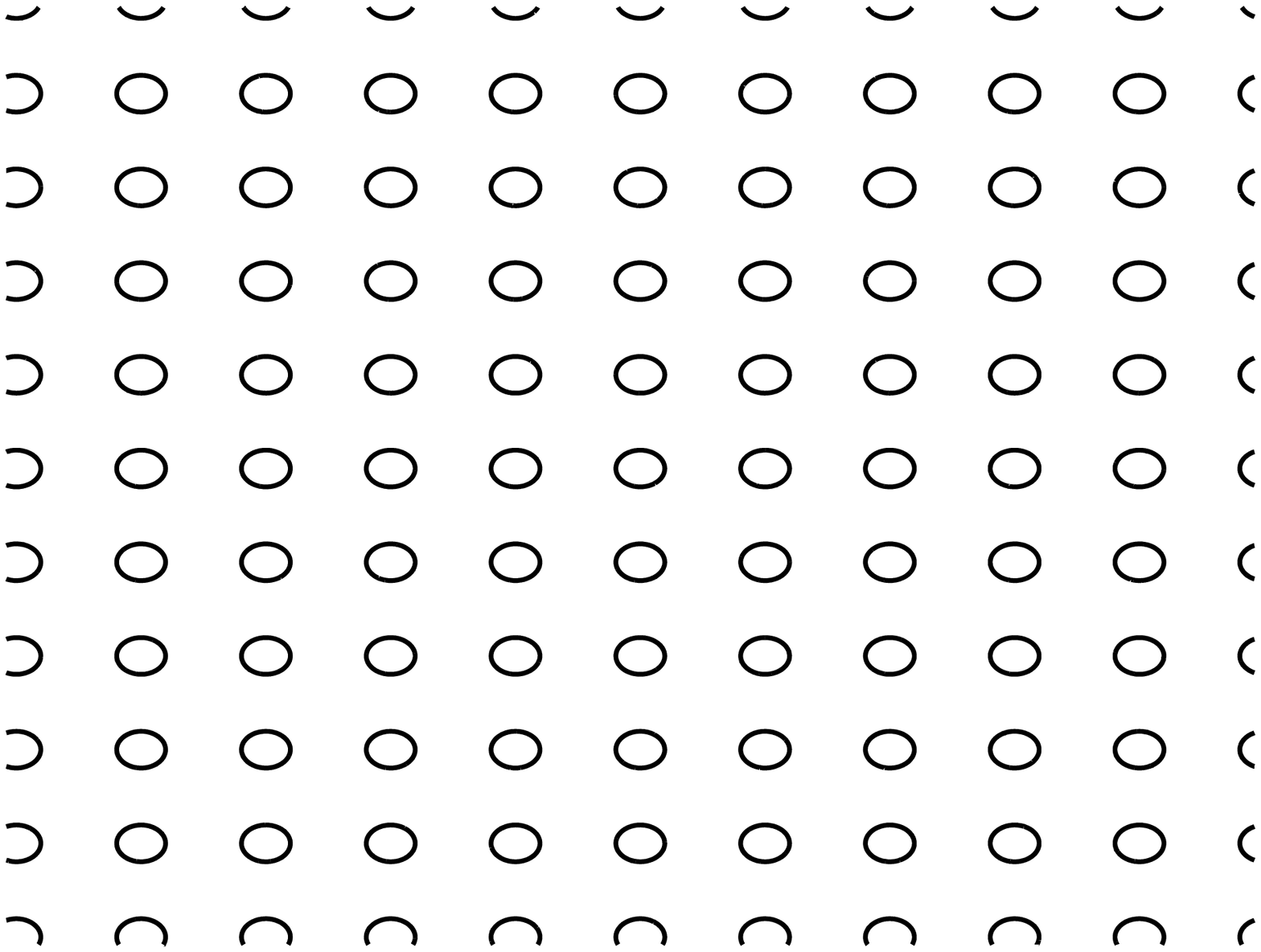}
\\
\includegraphics[width=0.5\textwidth]{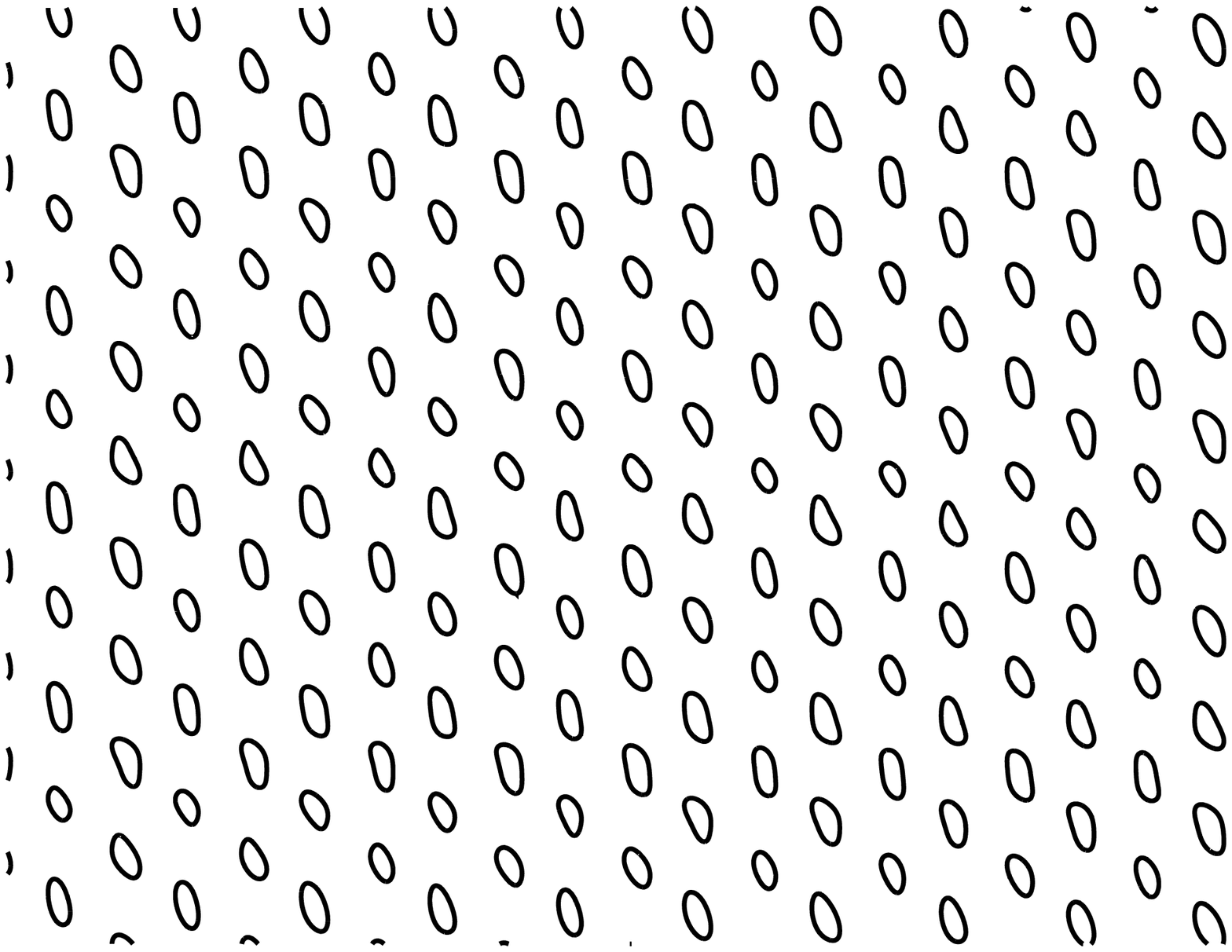}
\end{center}
\end{figure}

We have a linear homogenization problem for second order elliptic equations and it is natural to apply the standard oscillating test function method; see e.g. Murat and Tartar \cite{murat}. The key step is to build oscillating test functions from the auxiliary problem (or ``cell problem"): for any fixed vector $p \in \R^d$, find $w_p = (w_p^+, w_p^-)$ such that
\begin{equation}
-\nabla \cdot A(y,\omega) \left(p + \nabla w^{\pm}_p(y,\omega) \right) = 0, \quad\quad \text{on}\quad \Phi(\R^\pm_d,\omega),
\label{eq:aux}
\end{equation}
with proper transmission conditions across the interfaces of $\Phi(\Rdm)$ and $\Phi(\Rdp)$; see \eqref{eq:auxil} for more precise formulation. Here, $y = x/\eps$ is the microscale variable, $\R^d = \Rdp\cup \Gamd \cup \Rdm$ is a decomposition of $\R^d$ formed by the periodic interfaces $\Gamd$, and $\Phi(\cdot,\omega): \R^d \to \R^d$ is a random deformation on $\R^d$; see Figure \ref{fig1} for the decomposition. If $\Phi$ is the identity mapping for all realization, we recover the periodic setting of \cite{Monsur}, and the above problem is indeed posed on the unit cell $Y= [0,1]^d$ with a transmission condition across $\Gamma$, which is the unit interface that separates $Y$ into $Y^-$ and $Y^+$. The unit cell is compact, and the natural space for the solution is $H^1_{\rm per}(Y^+) \times H^1(Y^-)$, i.e. $w^-_p \in H^1(Y^-)$, $w^+_p \in H^1(Y^+)$ and $w^+_p$ satisfies periodic conditions at the boundary $\partial Y$. This space enjoys a Poincar\'e inequality and the existence and uniqueness of the cell problem is more or less standard. In the general random case, the cell problem is posed on a non-compact domain and its solution, for each realization $\omega$, lives in $H^1_{\rm loc}(\Phi(\Rdp)) \times H^1_{\rm loc}(\Phi(\Rdm))$ which does not admit any Poincar\'e inequality. This loss of compactness is a main difficulty in stochastic homogenization. In the standard stationary ergodic setting, e.g. in Papanicolaou and Varadhan \cite{PV79} and Kozlov \cite{Kozlov}, where the coefficients are stationary with respect to all $\R^d$ vectors, it is possible to lift the auxiliary problem to the probability space, because there is a natural correspondence between the physical derivative $\nabla$ and the infinitesimal generator of the translation operator $\tau_x$, $x \in \R^d$, on $\Omega$. One can solve the lifted problem in the probability space, owing to the Weyl decomposition of $L^2$ vectors in $\Omega$, and then push this solution back to the physical space. We refer to the aforementioned references for the details; see also Chapter 7 of \cite{Jikov_book}.

The stationary ergodic setting in this paper is not standard: stationarity is with respect to all $\Z^d$ translations (rather than $\R^d$ translations) only. The program above breaks down since it is not clear how to lift the problem to probability space. We need to solve \eqref{eq:aux} in the physical space. We first regularize the problem by adding a small zero-order absorption term, to remedy the lack of Poincar\'e inequality. Then we solve the regularized problem in the space of locally uniform Sobolev spaces for each realization, by solving a truncated Dirichlet problem on bounded domains that exhaust the whole space; our method takes inspiration from the work of G\'erard-Varet and Masmoudi \cite{GV_Masmoudi} and that of Dalibard and Prange \cite{Dalibard_Prange} where boundary layer systems of Navier-Stokes equations in unbounded channels were studied. Usage of locally uniform Sobolev spaces to study problems on unbounded domains was pioneered by Kato \cite{Kato}. The regularized solutions turn out to be stationary and we establish averaged (in the probability space) estimates for the gradient of the regularized solution that are uniform with respect to the regularization parameter. These uniform estimates allow us to obtain, when the regularization is sent to zero, a converging subsequence along which the gradient and the jumps on the interfaces converge to a stationary fields. We find the desired solution to \eqref{eq:auxil} from those limits. Once this is completed, the homogenization theory is proved essentially by the standard oscillating test function method. Sublinearity of the solution of the cell problem is a subtle issue as always. We prove this property by invoking elliptic regularity for the gradient of the solution up to the interfaces. For $d = 2,3$, $A \in C^1$ and $\Gamma \in C^2$ are sufficient.

The rest of this paper is organized as follows. In Section \ref{sec:settings} we make precise the problem settings on the diffusion coefficients and the interfaces and state the main results. In Section \ref{sec:prelim} we record some preliminary results, which includes ergodic theorems for the special random setting of \cite{BLBL06}, the random extension operator of \cite{AGGJS13} and basic energy estimates for \eqref{eq:aniso}. The proof of our main theorem is detailed in Section \ref{sec:proof} using the standard oscillating test function method of Murat and Tartar \cite{murat} while the key to the proof, i.e. the study of the auxiliary problem, is in Section \ref{sec:auxil}. Finally, we conclude the paper in Section \ref{sec:discussion} by showing some properties of the homogenized problem and by discussing some possible further studies. 

{\bf Notations.} We use the standard notation $(\Omega,\F,\Pb)$ for probability spaces: $\omega \in \Omega$ is a realization, and $\E$ is the mathematical expectation with respect to the probability measure $\Pb$. The standard notations $L^p$ and $W^{k,p}$ are used for the Lebesgue space and the Sobolev space respectively, and $H^k$ is used as a short-hand notation for $W^{k,2}$. The spatial domain for these functional spaces are usually specified. $W_{\rm loc}^{k,p}$ and $W_{\rm uloc}^{k,p}$ denote, respectively, local Sobolev and locally uniform Sobolev spaces. When there is no risk of confusion, Einstein's summation convention is adopted so repeated indices are summed over, e.g. $a_{ij} \xi_i \xi_j = \sum_{i=1}^d \sum_{j=1}^d a_{ij}\xi_i\xi_j$. For any vector $\xi$ in $\R^d$ or in $\Z^d$, the symbol $|\xi| = (\sum_{j=1}^d \xi_j^2)^{\frac 1 2}$ denotes the Euclidean norm and $|k|_\infty = \max_{1\le j\le d} |k_j|$ denotes the supremum norm. We write $U\subset \subset V$ to mean $U$ is compactly contained in $V$. For any measurable subset $A$ of $\R^d$, $|A|$ denotes its Lebesgue measure. Finally, if $S$ is a smooth $d-1$ dimensional surface in $\R^d$, $d\sigma$ denotes the standard induced Lebesgue measure on the surface.
\section{Problem Settings and Main Results}
\label{sec:settings}

In this section we first describe the random settings for the elliptic coefficients in \eqref{eq:aniso} and the interfaces which divide the spatial domain. Then we state the main results of the paper.
\subsection{Random ensemble of surfaces}
\label{sec:form}

The random medium of this paper, i.e. the random coefficients and the random interfaces in \eqref{eq:aniso}, is obtained as the image of a periodic medium with periodic coefficients and periodic interfaces under a random deformation followed by a rescaling. Hence, we describe the periodic setting first.

{\bf Periodic setting.} Due to periodicity, the medium is determined on the unit cell. This cell is denoted by $Y = [0,1)^d$, the unit cube in $\R^d$. Let $Y^-$ be a simply connected open subset of $Y$ with smooth (say $C^2$ for $d=2,3$) boundary $\partial Y^-$. Hence this boundary is the unit interface and is denoted as $\Gamma_0$. This interface decomposes the unit cell to $Y^-$ and $Y^+ := Y\setminus \overline{Y^-}$. Therefore, $Y^-$, $\Gamma_0$ and $Y^+$ represent, respectively, the unit interior region, the separating surface and the outer environment. Set $\beta = \dist(\partial Y, \Gamma_0)$ and assume that $\beta \lesssim 1$, that is $\beta$ is smaller than but comparable to one. To build a periodically structure, we set for all $k \in \Z^d$,
$$
Y_k = Y + k, \quad Y^+_k = Y^+ + k, \quad Y^-_k = Y^- + k, \quad \Gamma_k = \Gamma_0 + k.
$$
The union of all separating surfaces is then written as $\Gamd = \cup_{k \in \Z^d} \Gamma_k$. The union of all interior regions enclosed by these surfaces is denoted by $\Rdm = \cup_{k\in \Z^d} Y^-_k$ and the outer environment is characterized as $\Rdp = \R^d\setminus (\Gamd \cup \Rdm)$, or equivalently $\Rdp = \cup_{k \in \Z^d}Y^+_k$. Clearly, $\Gamd$ and $\Rdpm$ are periodic. Further, $\Rdp$ is connected while $\Rdm$ has simply connected components that are separated by a distance that is at least $2\beta$.

In addition to the geometry of the periodic medium, we specify its physical properties. We assume that the interior region $\Rdm$ and the environment $\Rdp$ are filled with a material whose conductivity is characterized by a matrix valued function $\Atilde(y) = (\widetilde{a}_{ij}(y))$. We assume further that 
\begin{enumerate}
\item[(A1)] $d = 2,3$ and $\Atilde$ is $[0,1)^d$-periodic, i.e. 
\begin{equation}
\widetilde{a}_{ij}(y+k) = \widetilde{a}_{ij}(y), \quad \text{ for all } y \in \R^d, k \in \Z^d,
\label{eq:perio}
\end{equation}
\item[(A2)] $\Atilde \in C^{1}$ and $\Atilde$ is uniformly elliptic, i.e. for some $0 < \lambda \le \Lambda$ it holds that
\begin{equation}
\lambda |\xi|^2 \le \widetilde{a}_{ij}(y)\xi_i \xi_j \le \Lambda |\xi|^2, \quad \text{ for all } y \in \R^d, \xi \in \R^d.
\label{eq:ellip}
\end{equation}
\end{enumerate}
The surfaces $\Gamd$ separate the materials occupying $\Rdp$ and $\Rdm$, and the physical properties of the surfaces are described by the transmission condition for the potential fields in $\Rdp$ and $\Rdm$ as seen in \eqref{eq:aniso}. The geometry and the physical properties together complete the periodic model medium with interfaces. This periodic medium is exactly the one studied by Monsurr\`o and her coauthors in \cite{Monsur,DM04,DFM07}.

{\bf Random setting.} Following the idea of Blanc, Le Bris and Lions \cite{BLBL06,BLBL07}, who considered random diffusive media where the conductivity tensor $A$ is obtained as the image of a periodic tensor $\widetilde{A}$ under a random deformation, we construct our random medium, i.e. the conductivity tensor and the conductivity resistant interfaces, by randomly deforming a periodic one. Let $\Phi: \R^d \times \Omega \to \R^d$ be a random orientation preserving diffeomorphism on $\R^d$; that is for each $\omega \in \Omega$, $\Phi(\cdot,\omega)$ is a diffeomorphism on $\R^d$. Then, for the periodic structure $(\Gamd,\Rdm)$ and $\widetilde{A}$ defined above, under each realization $\Phi(\cdot,\omega)$, one obtains the deformed structure $\Phi(\Rdp)\cup \Phi(\Gamd)\cup \Phi(\Rdm)$ with conductivity coefficient $\Atilde \circ\Phi^{-1}$. Again, the physical importance of the interfaces will appear as a transmission condition for the potential fields across them. We refer to this medium as the reference random medium. Note that $\Phi(\Rdp)$ remains connected, and $\Phi(\Rdm)$ has connected components.

To model the heterogeneous medium whose structure and physical properties vary on a small scale of $\eps$, $0<\eps \ll 1$, we rescale the reference random medium. This is done by using the scaling operator $\eps Id: \R^d \to \R^d$ given by $y \mapsto \eps y$. Consequently, we obtain a connected environment $\eps \Phi(\Rdp)$, the separating interfaces $\eps \Phi(\Gamd)$ and the interior regions $\eps \Phi(\Rdm)$. Besides, the materials in $\eps \Phi(\Rdpm)$ have conductivity coefficient $A^\eps := \Atilde\circ\Phi^{-1}(\frac{\cdot}{\eps})$.

Finally, in the open bounded domain $D$ on which \eqref{eq:aniso} is posed, we would like to set $\Depspm = D\cap \eps\Phi(\Rdpm)$ and $\Gamma_\eps = D\cap \eps\Phi(\Gamd)$. However, $\partial D \cap \eps\Phi(\Gamd)$ may not be empty. In other words, the boundary $\partial D$ may cut certain components of $\eps \Phi(\Gamd)$. In \cite{AGGJS13}, \eqref{eq:aniso} models diffusion phenomena in a suspension of cells and we would like to avoid the cells being cut by the boundary of the domain. This requires a modification of the above proposal of $\Depspm$ near the boundary of $D$. In this paper and with this biological application in mind, we keep this constraint though it can be removed as long as the intersection of $\partial D$ and the interfaces makes sense in the physical application. We provide the details of this modification in the next subsection under some assumptions on the diffeomorphsim $\Phi(\cdot,\omega)$. 

\subsection{Stationary and ergodic deformations}
\label{sec:diffe}

Let $\Phi$ be the aforementioned random orientation preserving diffeomorphism of $\R^d$ defined on some probability space $(\Omega,\F,\Pb)$. Throughout the paper, we assume that $\F$ is countably generated so that $L^2(\Omega)$ is separable. We assume further that the probability space has the following structure.
\begin{enumerate}
\item[(S1)] The group $(\Z^d,+)$ acts on $\Omega$ by some action $\{\tau_k : \Omega \to \Omega\}_{k \in \Z^d}$. For all $k \in \Z^d$, the map $\tau_k$ is $\Pb$-preserving, i.e. $\Pb(\tau_k A) = \Pb(A)$ for all $A \in \F$.
\item[(S2)] The group action above is ergodic, i.e. $A \in \F$ and $\tau_k A = A$ for all $k\in \Z^d$ implies that $\Pb(A) \in \{0,1\}$.
\end{enumerate}

In this paper, we say that a locally integrable random process $F \in L^1_{\rm{loc}}(\R^d,L^1(\Omega))$ is stationary if for all $x \in \R^d$ and $\omega \in \Omega$ we have
\begin{equation}
\forall k \in \Z^d, \quad F(x+k, \omega) = F(x,\tau_k \omega).
\label{eq:stati}
\end{equation}
This notion of stationarity for random processes is different from the standard one e.g. in \cite{Jikov_book,Kozlov,PV79} where the action $\{\tau_z\}_{z\in \R^d}$ is used and \eqref{eq:stati} should be satisfied for all $k \in \R^d$. It is known that neither of the two notions is a special case of the other; nevertheless, both notions include periodic functions as a special case. Nevertheless, it is well known that, e.g. as shown in \cite{BLBL06,BLBL07}, such stationary processes still enjoy certain types of ergodic theorems.

The main assumptions on the random diffeomorphism $\Phi$ are: for all $\omega \in \Omega$,
\begin{enumerate}
\item[(T1)] The random field $\nabla \Phi(y,\omega)$ is stationary.
\item[(T2)] There exists a constant $\mu$ such that $\inf_{y \in \R^d} \det(\nabla \Phi(y,\omega)) \ge \mu > 0$.
\item[(T3)] There exists a constant $M < \infty$ such that $\sup_{y \in \R^d} |\nabla \Phi(y,\omega)| \le M < \infty$ and $\sup_{y \in \R^d} |D^2 \Phi(y,\omega)| \le M$.
\end{enumerate}
We call any $\Phi$ satisfying the above conditions a stationary random diffeomorphism. Let $\Psi$ be the inverse of $\Phi$. Then (A2) and (A3) implies that for all $\omega \in \Omega$,
\begin{equation}
\sup_{x \in \R^d}|\nabla \Psi(x, \omega)| \le M' < \infty, \quad \inf_{x\in \R^d} \det(\nabla \Psi(x,\omega)) \ge \mu' > 0.
\label{eq:Psibdd}
\end{equation}
Here, $\mu'$ and $M'$ are two constants depending on $\mu, M$ and the dimension $d$ but not on $\omega$. By the uniform Lipschitz assumption (A3), for any two points $y_1, y_2 \in \R^d$, we have $\Phi(y_1) - \Phi(y_2) \le M|y_1 - y_2|$ and similarly by \eqref{eq:Psibdd}, we have $
|\Phi(y_1) - \Phi(y_2)| \ge (M')^{-1}|y_1 - y_2|$. These estimates indicate that in the reference random medium, the interfaces are still well separated at least by a distance of $2\beta/M'$. After the rescaling, the interfaces in $\Gameps$ are well separated by a distance of $2\eps \beta/M'$.

Now we construct $\Gameps$ more carefully so that $\partial D$ does not cut any component of $\Gameps$. For each $\omega \in \Omega$, let $\widetilde{D}_{\eps^{-1}} = \Phi^{-1}(\frac{D}{\eps})$ be the preimage of $D$ under the map $\eps\Phi(\cdot,\omega)$. Let $\widetilde{D}'_{\eps^{-1}}$ be its subset that is $\beta$ away from the boundary, i.e.
$$
\widetilde{D}'_{\eps^{-1}} = \{z \in \widetilde{D}_{\eps^{-1}} ~|~ \dist(z, \partial \widetilde{D}_{\eps^{-1}}) \ge \beta \}.
$$
Let $\Ieps \subset \Z^d$ be the indices of the cubes inside $\widetilde{D}'_{\eps^{-1}}$, i.e. those of $\{Y_k \subset \widetilde{D}'_{\eps^{-1}}\}$. Then we set
\begin{equation}
\Gameps = \sum_{k\in \Ieps} \eps \Phi(\Gamma_k), \quad \Depsm = \sum_{k\in \Ieps} \eps \Phi(Y^-_k), \quad \Depsp = D\setminus \overline{\Depsm}.
\label{eq:Depsdef}
\end{equation}
That is, we only keep the deformed and rescaled cells that are inside $D$ and have a distance at least $\eps \beta/M'$ away from the boundary. We also define the following two subsets of $D$:
\begin{equation}
E_\eps = \sum_{k \in \mathcal{I}_\eps} \eps \Phi(Y_k), \quad\quad  K_\eps = D \setminus \overline{E_\eps}.
\label{eq:KEdef}
\end{equation}
The set $E_\eps$ encloses all the $\eps$-scale interfaces in $\Gameps$, the region inside these surfaces, i.e. $\Depsm$ and their immediate surroundings $\cup_{k \in \Ieps} \eps\Phi(Y_k^+)$. The set $K_\eps$ can be thought as a cushion layer close to the boundary that prevents the interfaces from touching the boundary. From the construction we verify that
\begin{equation}
\inf_{x \in \Depsm} \mathrm{ dist }(x,\partial D) \ge \eps \beta/M', \quad \text{ and }\quad  \sup_{x \in K_\eps} \mathrm{ dist }(x, \partial D) \le \eps \beta \sqrt{d}M.
\end{equation}
Hence, the interfaces $\Gameps$ are separated from $\partial D$ and the cushion layer is restricted to a vicinity of $\partial D$ whose thickness is comparable to $\eps$.

\begin{rem} We provide some examples. First if $\Phi = Id$ is the identity operator, we recover the periodic setting. If $\Phi$ is a deterministic diffeomorphism, we obtain a deterministic deformed medium. For a less trivial example, let $\bm{X} = \{X_k \;|\; k \in \Z^d\}$ be the set of i.i.d. Bernoulli variables with indices in $\Z^d$, i.e. each $X_k$ is either $0$ or $1$ with probability $\frac{1}{2}$. Set the probability space $(\Omega,\F,\Pb)$ to be the canonical space for the random process $\bm{X}$. That is, $\Omega = \{0,1\}^{\Z^d}$; $\F$ is the Borel $\sigma$-algebra generated by finite dimensional cylindrical sets in $\Omega$ and $\Pb$ is defined by setting, for any $A \in \F$,
$$
\Pb(A) = \Pb_0\{\bm{X} \in A\}
$$
where $\Pb_0$ is the underlying probability measure associated to the Bernoulli sequence. We then check that the group $\{\tau_k \;|\; k \in \Z^d\}$ which acts on $\Omega$ by 
\begin{equation*}
\tau_k \bm{X} = \tau_k \{X_{\ell} \;|\; \ell \in \Z^d\} = \{X_{\ell + k} \;|\; \ell \in \Z^d\}
\end{equation*}
is measure preserving and ergodic.

Now consider two smooth functions $\Phi_{0,1}: Y \to Y$ given by: $\Phi_{0} = Id$; $\Phi_1 - Id \not \equiv 0$ and $\Phi_1 - Id$ is compactly supported in $Y$, $|\nabla \Phi_1| \le M$ and $\mathrm{det}(\nabla \Phi_1) \ge \mu$. For each $\omega \in \Omega$, i.e. for each $\bm{X} = \{X_k \;|\; k\in \Z^d\}$ and for each $x \in \R^d$, with $[x]$ denoting the unique number in $\Z^d$ such that $x - [x] \in [0,1)^d$, we set
$$
\Phi(x,\omega) = [x] + \Phi_{X_{[x]}}(x-[x]).
$$
Then $\Phi(x,\omega)$ is a random diffeomorphism satisfying the aforementioned conditions. One checks that for each cube $Y_k$, $\Phi$ leaves its boundary unchanged and may or may not deform its interior according to the outcome of the Bernoulli variable $X_k$.
\end{rem}
\subsection{The main results}
\label{sec:results}

{\bf Assumptions (A).} Throughout the rest of this paper, we assume that the random coefficients $A^\eps$, the random surfaces $\Gameps(\omega)$, and the decomposition of $D$ into $\Depsp$ and $\Depsm$ in \eqref{eq:aniso} are constructed as in Section \ref{sec:form}. We assume that the hypothesis (A1)(A2) on the coefficient $A$, (S1)(S2) on the probability space and (T1)(T2)(T3) on the random diffeomorphism hold. We assume that $f \in L^2(D)$ in \eqref{eq:aniso}.

Due to the jump type transmission condition across the interfaces $\Gameps$, the solution $u_\eps$ that solves \eqref{eq:aniso} are piecewisely defined as $u_\eps^+$ on $\Depsp$ and $u_\eps^-$ on each components of $\Depsm$. A natural space for the solution is $H^1(\Depsp(\omega)) \times H^1(\Depsm(\omega))$. As a result, the functional space on which the solutions are defined depends on both $\eps$ and $\omega$. This poses some difficulty on making sense of the convergence of $\uepspm$. Hence for $\uepsp$ we introduce certain extension $\uepsp(\cdot,\omega) \mapsto \uepsext(\cdot, \omega)$ where the latter function belongs to $H^1(D)$ and agrees with $\uepsp$ on $\Depsp$; see Proposition \ref{prop:ext_Roe} below. For $\uepsm$, we take the trivial extension $\uepsm(\cdot,\omega) \mapsto Q\uepsm(\cdot,\omega)$ where the latter belongs to $L^2(D)$ and vanishes on $\Depsp$.

The main result of this paper is the almost sure homogenization of the problem \eqref{eq:aniso}. 
We first introduce two quantities that appear in the presentation of the homogenized problem. They are $\varrho$, the mean volume of the unit cube $Y$ after deformation, and $\theta$, the mean volume fraction of $Y^-$ after deformation. They are given by
\begin{equation}\label{eq:rho_theta}
\begin{aligned}
\varrho &= \E \int_{Y} \det \nabla \Phi(z) dz = \E |\Phi(Y,\omega)|,\\
\theta &= \frac{1}{\varrho}\ \E \int_{Y^-} \det \nabla \Phi(z) dz = \frac{\E |\Phi(Y^-,\omega)|}{\E |\Phi(Y,\omega)|}
\end{aligned}
\end{equation}
Due to the assumptions on $\Phi$, we verify that $0 < \theta < 1$.

Our first main result is on the solution to the auxiliary problem (cell problem) \eqref{eq:aux}. We recall that by construction, $A = \Atilde \circ\Psi$ where $\Atilde$ is the periodic coefficient, and $\Psi$ is the inverse of the random diffeomorphism. Recall that the assumptions at the beginning of this subsection are invoked.

\begin{theorem}\label{thm:auxil} For each $p \in \R^d$, there exists a function $w_p(y,\omega)$ defined on $\R^d\times \Omega$, of the form $\widetilde{w}_p(\Psi(y,\omega),\omega)$ with $\widetilde{w}_p = (\widetilde{w}^+_p,\widetilde{w}_p^-) \in L^2(\Omega,H^1_{\rm{loc}}(\Rdp) \times H^1_{\rm{loc}}(\Rdm))$ and satisfies: $(\nabla \widetilde{w}^+_p,\nabla \widetilde{w}^-_p)$ is stationary, $\widetilde{w}^+_p$ is the restriction of $\widetilde{w}^{\rm ext}_p \in H^1_{\rm loc}(\R^d)$ in $\Rdp$ whose gradient $\nabla \widetilde{w}^{\rm ext}_p$ is stationary and satisfies
\begin{equation}
\E \int_{Y} \nabla \widetilde{w}^{\rm ext}_p(y,\omega) dy = 0.
\label{eq:auxil_avg}
\end{equation}
Moreover, for a.e. $\omega \in \Omega$, $w_p(\cdot,\omega)$ sublinear at infinity and $w_p(\cdot,\omega)$ is the weak solution to the following problem
\begin{equation}
\left\{
\begin{aligned}
&-\nabla \cdot\left( A (y,\omega) \nabla \left[w_p^\pm(y,\omega) + p\cdot y\right] \right) = 0, &\quad &\text{ in } \  &\Phi(\Rdpm,\omega),\\
&\frac{\partial}{\partial \nu_A} w_p^+(y,\omega) = \frac{\partial}{\partial \nu_A} w_p^-(y,\omega), &\quad &\text{ on } \  &\Phi(\Gamd,\omega),\\
&w_p^+(y,\omega) - w_p^-(y,\omega) = \frac{\partial}{\partial \nu_A} w_p^+(y,\omega) + \nu_y \cdot Ap, &\quad &\text{ on } \  &\Phi(\Gamd,\omega).
\end{aligned}
\right.
\label{eq:auxil}
\end{equation}
The solution $w_p = (w^+_p,w^-_p)$ is unique up to an additive constant $C(\omega)$.
\end{theorem}

We note the problem is posed on the whole space $\Phi(\Rdp,\omega) \cup \Phi(\Gamd,\omega) \cup \Phi(\Rdm,\omega)$, with infinitely many interfaces in $\Phi(\Gamd,\omega)$. We emphasize again that, due to the lack of compactness and the $\Z^d$-stationarity of $\nabla \Phi$, solving this auxiliary problem is nontrivial.

For the main result on the homogenization theory of \eqref{eq:aniso}, we define $A^0 = (a^0_{ij})$,  the homogenized conductivity coefficients by
\begin{equation}
a^0_{ij} : = \frac{1}{\varrho}\ \E\left( \int_{\Phi(Y^+,\omega)} e_j \cdot A(e_i + \nabla w^+_{e_i}) dx + \int_{\Phi(Y^-,\omega)} e_j \cdot A (e_i + \nabla w^-_{e_i}) dx\right).
\label{eq:Abardef}
\end{equation}

\begin{theorem} \label{thm:homog} Let $A^0 = (a^0_{ij})$ be a deterministic constant matrix defined above and let $u_0 \in H^1(D)$ be the unique solution to \eqref{eq:homog}. There exists a subset $\Omega_* \subset \Omega$ with $\Pb(\Omega_*) = 1$, and for each $\omega \in \Omega_*$, the sequence of unique solutions $u_\eps(\cdot, \omega)$ to \eqref{eq:aniso} satisfy that as $\eps \to 0$,
\begin{enumerate}
\item[\upshape(i)] The function $u_\eps$ converges strongly in $L^2(D)$ to $u_0$.
\item[\upshape(ii)] The extended function $\uepsext(\cdot,\omega) \in H^1_0(D)$ of $\uepsp(\cdot,\omega)$ given by the extension operator of Proposition \ref{prop:ext_D} converges weakly in $H^1(D)$ to $u_0$.
\item[\upshape(iii)] The trivial extension $Q \uepsm(\cdot,\omega)$ of $\uepsm(\cdot,\omega)$ converges weakly in $L^2(D)$ to $\theta u_0$.
\item[\upshape(iv)] The flux $\chi_{D_\eps^+} A^\eps \nabla \uepsp + \chi_{D_\eps^-} A^\eps \nabla \uepsm$ converges weakly in $(L^2(D))^d$ to the homogenized flux $A^0\nabla u_0$.
\end{enumerate} 
\end{theorem}

The homogenized equation \eqref{eq:homog} has unique solution; in fact one can verify that the homogenized conductivity $A^0$ is uniformly elliptic; see Section 6.

\section{Preliminary Results}
\label{sec:prelim}

We recall some facts that will be used later. These include properties of ergodic processes in the sense of \eqref{eq:stati}, functional spaces defined on ``perforated" domains and extension operators, and some basic energy estimates for \eqref{eq:aniso}.

\subsection{Stationary ergodic random processes}
\label{sec:stati}

As mentioned earlier, the notion of stationarity in this paper is different from the standard one, e.g. in \cite{PV79,Kozlov,Jikov_book}. Nevertheless, the following version of ergodic theorems (see e.g. Dunford and Schwartz \cite{Dunford} and Krengel \cite{Krengel}) hold.
\begin{proposition}\label{prop:ergod}
{\upshape(i)} Let $F \in L^p(\Omega,L^p_{\rm loc}(\R^d))$, $p \in [1,\infty)$, be a stationary random process in the sense of \eqref{eq:stati}. Then
\begin{equation}
F\left(\frac{z}{\eps}, \omega \right)  \ \xrightharpoonup[\eps \to 0]{L^p_{\rm loc}} \ \E\left(\int_Y F(z,\cdot) dz \right), \quad \text{ for a.e.}\ \omega \in \Omega.
\end{equation}

{\upshape(ii)} If $F \in L^\infty_{\rm loc}(\R^d,L^1(\Omega))$ is a stationary random process, then the above convergence holds in $L^\infty$ weak-$*$ for a.e. $\omega \in \Omega$.
\end{proposition}

The above conclusions in the $\R^d$-stationary setting can be found in Theorem 7.2 of \cite{Jikov_book} and the $\Z^d$-stationary setting is of the same spirit. Since we mainly deal with functions on the deformed space, we will encounter functions which are not stationary themselves but their preimage before the deformation is. Such a function can be written as $g\circ\Psi(y,\omega)$ where $g$ is stationary. We have the following result.

\begin{lemma}\label{lem:gstat} {\upshape(i)} Let $g \in L^p(\Omega, L^p_{\rm loc}(\R^d))$, $p \in (1,\infty)$, be a stationary process in the sense of \eqref{eq:stati}. Let $\Psi$ be defined as in section \ref{sec:diffe}. Then we have
\begin{equation}
\label{eq:gstat}
g\left(\Psi\left(\frac{x}{\eps},\omega\right),\omega\right) \xrightharpoonup[\eps \to 0]{L^p_{\rm loc}} \frac{1}{\varrho}\ \E \left(\int_{\Phi(Y,\cdot)} g\circ \Psi(y, \cdot) dy \right) \quad \text{ for a.e. } \omega \in \Omega.
\end{equation}

{\upshape(ii)} If $g \in L^\infty_{\rm loc}(\R^d,L^1(\Omega))$ be a stationary field. Then the above convergence result holds in $L^\infty$ weak-$*$ for a.e. $\omega \in \Omega$.
\end{lemma}

Item (ii) above was proved by Blanc, Le Bris and Lions \cite[Lemma 2.2]{BLBL06}. The proof for (i) is essentially the same. We provide it here for the sake of completeness.

\begin{proof} For any $p \in (1,\infty)$, let $p'$ be the H\"older conjugate of $p$. In view of the density of simple functions in $L^{p'}(K)$, for any bounded regular measurable set $K \subset \R^d$, it suffices to show that for a.e. $\omega \in \Omega$, for all such $K$,
\begin{equation}\label{e.statik}
\begin{aligned}
& \int_K g\circ \Psi\left(\frac{x}{\eps},\omega\right)dx = \int_{\eps \Phi^{-1}\left(\frac{K}{\eps}\right)} g\left(\frac{z}{\eps},\omega\right) \mathrm{det} \left(\nabla \Phi\left(\frac{z}{\eps},\omega\right)\right) dx\\
= &\int_{\R^d} \chi_{\eps \Phi^{-1}\left(\frac{K}{\eps}\right)} g\left(\frac{z}{\eps},\omega\right) \mathrm{det} \left(\nabla \Phi\left(\frac{z}{\eps},\omega\right)\right) dx \xrightarrow{\eps \to 0} \frac{|K|}{\varrho} \E \int_{\Phi(Y)} g\circ \Psi(y,\omega) dy.
\end{aligned}
\end{equation}
It is proved in Lemma 2.1 of \cite{BLBL06} that $\eps \Phi^{-1}\left(\frac{x}{\eps},\omega\right)$ converges to $\E \left[\int_Y \nabla \Phi(y,\cdot)dy\right]^{-1} x$ locally uniformly  as $\eps \to 0$, and the indicator of the set $\eps \Phi^{-1}\left(\frac{K}{\eps}\right)$ converges strongly in $L^{p'}$ to that of $\E \left[\int_Y \nabla \Phi(y,\cdot)dy\right]^{-1} K$. Applying Proposition \ref{prop:ergod} to the stationary random field $g \det(\nabla \Phi)$, which in view of the bounds on $\det (\nabla \Phi)$ is in $L^p_{\rm loc}(\R^d)$, we have that for a.e. $\omega$,
$$
g\left(\frac{z}{\eps},\omega\right) \mathrm{det} \left(\nabla \Phi\left(\frac{z}{\eps},\omega\right)\right) \xrightharpoonup[\eps \to 0]{L^{p}_{\rm loc}} \E \int_Y g(z,\omega) \mathrm{det} \left(\nabla \Phi(z,\omega)\right) dz = \E \int_{\Phi(Y)}g\circ \Psi(y,\omega) dy.
$$
Hence the integrand in the third integral in \eqref{e.statik} is a product of a term that converges strongly in $L^{p'}$ with a term that converges weakly in $L^p$. We hence have
\begin{equation*}
\begin{aligned}
\int_K g\circ \Psi\left(\frac{x}{\eps},\omega\right)dx \xrightarrow{\eps \to 0} &\int_{\E \left[\int_Y \nabla \Phi(y,\cdot)dy\right]^{-1} K} \left(\E \int_{\Phi(Y)}g\circ \Psi(y,\omega) dy\right) dx \\
= \quad& |K| \mathrm{det}\left(\E \left[\int_Y \nabla \Phi(y,\cdot)dy\right]^{-1} \right)\left(\E \int_{\Phi(Y)}g\circ \Psi(y,\omega) dy\right)\\
= \quad& |K| \left(\mathrm{det}\ \E \left[\int_Y \nabla \Phi(y,\cdot)dy\right] \right)^{-1}\left(\E \int_{\Phi(Y)}g\circ \Psi(y,\omega) dy\right).
\end{aligned}
\end{equation*}
In particular, if we set $g \equiv 1$ and recall the definition of $\varrho$ in \eqref{eq:rho_theta}, we get
$$
\left(\mathrm{det}\ \E \left[\int_Y \nabla \Phi(y,\cdot)dy\right] \right)^{-1}\left(\E \int_{\Phi(Y)}1 \ dy\right) = \left(\mathrm{det}\ \E \left[\int_Y \nabla \Phi(y,\cdot)dy\right] \right)^{-1} \varrho =   1.
$$
Substitute this relation to the preceding equation; we obtain the desired result \eqref{e.statik} and hence completes the proof.
\end{proof}

Another useful fact about random processes with stationary gradients is that they grow sublinearly at infinity. We state this result in the lemma below. The same conclusion for the $\R^d$-stationary setting was proved following the same argument of Lemma A.5 in \cite{armstrong}; see also Theorem 9 in \cite{Kozlov85}.

\begin{lemma}\label{lem:subli} Suppose that $w: \R^d \times \Omega \to \R$ and for almost every $\omega \in \Omega$, $G = Dw$ in the sense of distribution. Assume that for some $\alpha > d$, $w(\cdot,\omega)$ belongs to $W^{1,\alpha}_{\rm loc}(\R^d)$. Suppose $G$ is stationary and satisfies
\begin{equation}
\E \int_Y G(z,\cdot) dz = 0, \quad \text{and}\quad \E \int_Y |G(z,\cdot)|^\alpha dz < \infty.
\end{equation}
Then
\begin{equation}
\lim_{|y| \to \infty} |y|^{-1} w(y,\omega) = 0, \quad \text{for a.e. } \omega \in \Omega.
\end{equation}
\end{lemma}

\subsection{Extension Lemmas}
\label{sec:funct}

We record in this section some extension operators for functions defined on $\Rdp \times \Rdm$, $\Phi(\Rdp) \times \Phi(\Rdm)$, and $\eps\Phi(\Rdp) \times \eps\Phi(\Rdm)$. The starting point is to introduce extension operators for functions defined on $Y^+ \times Y^-$, inside the unit cell. We have

\begin{theorem}\label{thm:ext_Y}
Let $Y^+, Y^-$ and $\Gamma_0$ be as defined in Section \ref{sec:form}. Then there exists an extension operator $P : W^{1,p}(Y^+) \to W^{1,p}(Y)$ for all $p \ge 1$, and a constant $C = C(d,p,Y^-)$ such that, for any $p \ge 1$ and any $f \in W^{1,p}(Y^+)$, we have
\begin{equation}\label{eq:ext_Y}
\|Pf\|_{L^p(Y)} \le C\|f\|_{L^p(Y^+)}, \quad \quad \|\nabla Pf\|_{L^p(Y)} \le C\|\nabla f\|_{L^p(Y^+)}.
\end{equation}
\end{theorem}

This theorem was first proved by Cioranescu and Saint Jean Paulin \cite{CSP}; see also the book of Zhikov, Kozlov and Ole\u{\i}nik \cite{Jikov_book}. The extension operator $P$ is given by
\begin{equation}\label{eq:Pdef}
Pf = \frac{1}{|Y^+|} \int_{Y^+} f dx + E \left(f - \frac{1}{|Y^+|} \int_{Y^+} f dx\right),
\end{equation}
where $E$ is the more standard extension operator for Sobolev functions on bounded domain; see e.g. Section 5.4 of \cite{Evans}. The subtraction of the averaged value of $f$ over $Y^+$ is needed to have the first inequality. In fact, for the standard extension operator $E$, one only has $\|E f\|_{W^{1,p}} \le C\|f\|_{W^{1,p}}$, which is not good enough for the scaling performed below.

For functions in $W^{1,p}_{\rm loc}(\Rdp)$, we apply the extension operator $P$ above in each cube $Y_k$, $k \in \Z^d$. Then we obtain an extension operator from $W^{1,p}_{\rm loc}(\Rdp)$ to $W^{1,p}_{\rm loc}(\R^d)$. Denote this extension operator still by $P$. Evidently, we have
\begin{proposition}\label{prop:ext_R}
Let $P : W^{1,p}_{\rm loc}(\Rdp) \to W^{1,p}_{\rm loc}(\R^d)$ be defined as above. Then for the same $C$ as in Theorem \ref{thm:ext_Y} and for any $K \subset \subset \R^d$, we have
\begin{equation}\label{eq:ext_RK}
\|Pf\|_{L^p(K)} \le C\|f\|_{L^p(K\cap \Rdp)}, \quad \quad \|\nabla Pf\|_{L^p(K)} \le C\|\nabla f\|_{L^p(K\cap \Rdp)}.
\end{equation}
\end{proposition}

Fix an $\omega \in \Omega$, for any $f(\cdot,\omega) \in W^{1,p}_{\rm loc}(\Phi(\Rdp))$, we define $P_\omega f$ as 
\begin{equation}\label{eq:Pdef_o}
\left(P_\omega f\right) (\cdot, \omega) = \left[P\left(f \circ \Phi \right)\right] \circ \Phi^{-1}(\cdot,\omega).
\end{equation}

\begin{proposition}\label{prop:ext_Ro}
Let $P_\omega : W^{1,p}_{\rm loc}(\Phi(\Rdp,\omega)) \to W^{1,p}_{\rm loc}(\R^d)$ be defined as above. Then there exists some $C = C(d,p,Y^-,M,\mu)$, which is independent of $\omega$, such that for any $K \subset \subset \R^d$, we have
\begin{equation}\label{eq:ext_RKo}
\begin{aligned}
\|P_\omega f\|_{L^p(\Phi(K,\omega))} &\le C\|f\|_{L^p(\Phi(K\cap \Rdp,\omega))},
\\
\|\nabla P_\omega f\|_{L^p(\Phi(K,\omega))} &\le C\|\nabla f\|_{L^p(\Phi(K \cap \Rdp,\omega))}.
\end{aligned}
\end{equation}
\end{proposition}

For a proof of this result, we refer to Appendix A of \cite{AGGJS13}. The constant $C$ above can be made independent of $\omega$ because the bounds in (A2)(A3) and \eqref{eq:Psibdd} are uniform in $\omega$. Next, we consider functions defined on the scaled space. For any $\omega \in \Omega$ and $f \in \eps\Phi(\Rdp)$, define
\begin{equation}\label{eq:Pdef_oe}
\left(P^\eps_\omega f\right) (\cdot, \omega) =  \left(P_\omega f_\eps(\cdot,\omega) \right)\left(\frac{\cdot}{\eps}\right), \quad \text{where} \quad f_\eps(x,\omega) = f (\eps x, \omega).
\end{equation}
Then we have
\begin{proposition}\label{prop:ext_Roe}
Let $P^\eps_\omega : W^{1,p}_{\rm loc}(\eps\Phi(\Rdp,\omega)) \to W^{1,p}_{\rm loc}(\R^d)$ be defined as above. Then there exists some $C = C(d,p,Y^-,M,\mu)$, which is independent of $\omega$, such that for any $K \subset \subset \R^d$, we have
\begin{equation}\label{eq:ext_RKoe}
\begin{aligned}
\|P^\eps_\omega f\|_{L^p(\eps\Phi(K,\omega))} &\le C\|f\|_{L^p(\eps\Phi(K\cap \Rdp,\omega))},\\ 
\|\nabla P^\eps_\omega f\|_{L^p(\eps\Phi(K,\omega))} &\le C\|\nabla f\|_{L^p(\eps\Phi(K\cap \Rdp,\omega))}.
\end{aligned}
\end{equation}
\end{proposition}

Finally, recall the decomposition of $D$ in \eqref{eq:KEdef}. We can extend a function $f \in W^{1,p}(\Depsp)$ to $P^\eps_\omega f \in W^{1,p}(D)$ by using the extension operator $P^\eps_\omega$ in \eqref{eq:Pdef_oe} on each of the deformed and rescaled cubes $\eps \Phi(Y_k,\omega)$ in $E_\eps$, while leaving the function unchanged in the cushion layer $K_\eps$. Abusing notations, we denote this operator still by $P^\eps_\omega$. Then
\begin{proposition}\label{prop:ext_D}
Let $P^\eps_\omega : W^{1,p}(\Depsp) \to W^{1,p}(D)$ be as above. Then there exists some $C = C(d,Y^-,M,\mu)$ such that, for all $f \in W^{1,p}(\Depsp)$,
\begin{equation}
\|P^\eps_\omega f\|_{L^p(D)} \le C\|f\|_{L^p(\Depsp)}, \quad\quad
\end{equation}
\end{proposition}

For the proofs of Propositions \ref{prop:ext_Roe} and \ref{prop:ext_D}, we refer to Appendix A of \cite{AGGJS13}. The periodic versions of these propositions were developed by Monsurr\`o \cite{Monsur}. 

\subsection{Basic energy estimates}
\label{sec:energ}

Here we record some basic energy estimates for the solutions of \eqref{eq:aniso}.
\subsubsection{Functional space on the perforated domain in $D$}

Fix an $\eps > 0$ and a realization $\omega \in \Omega$,
the natural functional space for \eqref{eq:aniso} is
\begin{equation}
\Wcal_\eps := \left\{ u = u^+\chi_\eps^+ + u^-\chi_\eps^- \,|\, u^+ \in
H^1(\Depsp), \, u^- \in H^1(\Depsm), \, u\lvert_{\partial D} = 0 \right\},
\end{equation}
where $\chi_\eps^{\pm}$ denote the characteristic functions of the
sets $D_\eps^{\pm}(\omega)$, and $u\lvert_{\partial D}$ is the trace of $u$ on $\partial D$.
It is easy to verify that
\begin{equation}\label{eq:W-norm}
\|u \|_{\Wcal_\eps} = \left(\|\nabla u^+ \|_{L^2(\Depsp)}^2 + \|\nabla
u^- \|_{L^2(\Depsm)}^2 + \eps \|u^+ -
u^-\|_{L^2(\Gameps)}^2\right)^{\frac 1 2}
\end{equation}
defines a norm on $\Wcal_\eps$. Abusing notations, we set $H^1_0(\Depsp) := \{w \in H^1(\Depsp) \,|\, u\rvert_{\partial D} = 0\}$, that is $H^1$ functions on $\Depsp$ that vanish at the boundary of $\partial D$. Then $\Wcal_\eps = H^1_0(\Depsp)\times H^1(\Depsm)$, and in view of the Poincar\'e inequality \eqref{eq:poincare}, the
somewhat more standard norm for this space is given by
\begin{equation}\label{eq:H1-norm}
\|u\|_{H_0^1(\Depsp)\times H^1(\Depsm)} = \left(\|\nabla
u^+\|_{L^2(\Depsp)}^2 + \|\nabla u^-\|_{L^2(\Depsm)}^2 +
\|u^-\|_{L^2(\Depsm)}^2 \right)^{\frac 1 2}.
\end{equation}
In fact, these two norms are equivalent:

\begin{proposition}\label{prop:equiv}
The norm $\|\cdot\|_{\Wcal_\eps}$ is equivalent with the standard norm
in \eqref{eq:H1-norm}. Moreover, there exist
positive constants $C_1 < C_2$, independent of $\eps$ and $\omega$, such that for all $u \in \Wcal_\eps$, we have
\begin{equation}
\label{eq:equiv} C_1 \|u\|_{\Wcal_\eps} \le \|u\|_{H_0^1(\Depsp) \times H^1(\Depsm)}
\le C_2 \|u\|_{\Wcal_\eps}.
\end{equation}
\end{proposition}

This equivalence relation was established by Monsurr\`{o}
\cite{Monsur} in the periodic setting, and in the random setting it was proved in \cite{AGGJS13}. Some additional properties of the functions in $\Wcal_\eps$ are recorded below; we refer to \cite[Appendix C]{AGGJS13} for the proof, and to \cite{Monsur} for similar results in the periodic setting.

\begin{proposition}
There exists some constant $C > 0$, which is independent of $\eps$ and $\omega$, such that for all $v \in \Wcal_\eps$, we have
\begin{eqnarray}
\|v^+\|_{L^2(\Depsp)} &\le& C\|\nabla v^+\|_{L^2(\Depsp)}, \label{eq:poincare}\\
\|v^-\|_{L^2(\Depsm)} &\le& C\left(\sqrt{\eps} \|v^-\|_{L^2(\Depsm)} + \eps \|\nabla v^-\|_{L^2(\Depsm)}\right), \label{eq:DmL2}\\
\|v\|_{L^2(D)} &\le& C\|v\|_{\Wcal_\eps}. \label{eq:DL2}
\end{eqnarray}
\end{proposition}

\subsubsection{The energy estimates}

With $\eps>0$ and $\omega \in \Omega$ fixed as before, a function $\ueps \in \Wcal_\eps$ is said to be a weak solution to \eqref{eq:aniso} if for all $v = v^+ \chi_\eps^+ + v^- \chi_\eps^- \in \Wcal_\eps$, the following holds:
\begin{equation}
\begin{aligned}
\int_{\Depsp} A^\eps(x,\omega) \nabla \uepsp(x) \cdot \nabla v^+(x) dx \, + \, \int_{\Depsm} A^\eps(x,\omega) \nabla \uepsm(x) \cdot \nabla v^-(x) dx\\
+  \, \frac{1}{\eps} \int_{\Gameps}
(\uepsp-\uepsm)(v^+ - v^-) d\sigma(x) = \int_D f(x)v(x) dx.
\end{aligned}
\label{eq:rpdeweak}
\end{equation}

By the standard Lax-Milgram theorem, one obtains the existence and uniqueness of the weak solution to \eqref{eq:aniso} and some basic energy estimates.

\begin{proposition}\label{prop:energy} Let $f \in L^2(D)$. There exists a unique weak solution $u_\eps \in \Wcal_\eps$ for \eqref{eq:aniso}. Moreover, assume that $\eps \le 1/\sqrt{2}$ there exists a constant $C$, which is independent of $\eps$ and $\omega$, such that
\begin{equation}
\label{eq:GradientEst} \|\nabla \uepsp\|_{L^2(\Depsp)} + \|\nabla
\uepsm\|_{L^2(\Depsm)} \le C \|f\|_{L^2},
\end{equation}
\begin{equation}
\|\uepsp - \uepsm\|_{L^2(\Gameps)} \le C \sqrt{\eps}\|f\|_{L^2}.
\label{eq:L2GammaEst}
\end{equation}
\end{proposition}

\begin{proof}
For the existence and uniqueness of weak solution, define the bilinear form $\cA^\eps(\cdot,\cdot)$ on $\Wcal_\eps \times \Wcal_\eps$ and the linear form $\ell$ on $\Wcal_\eps$ by
\begin{equation*}
\cA^\eps(u,v) := \int_{\Depsp} A^\eps \nabla u^+
\cdot \nabla v^+ + \int_{\Depsm} A^\eps \nabla u^-
\cdot \nabla v^- + \frac{1}{\eps}
\int_{\Gameps} (u^+ - u^-)(v^+ - v^-) d\sigma(x),
\end{equation*}
\begin{equation*}
\ell(v) := \int_D f(x) v(x) dx.
\end{equation*}
It is clear that $\cA^\eps$ and $\ell$ are bounded operators. Moreover, due to \eqref{eq:ellip}, we have
\begin{equation*}
\cA^\eps(u,u) \ge \min(\lambda,1) \|u\|_{\Wcal_\eps}^2.
\end{equation*}
This shows that $\cA^\eps$ is coercive. By Lax--Milgram theorem
there exists a unique $u_\eps \in \Wcal_\eps$ such that $\cA^\eps(u_\eps,v) = \ell(v)$ for all $v \in \Wcal_\eps$, i.e. $u_\eps$ is the weak solution to \eqref{eq:aniso}.

To obtain the energy estimates, we take $v = u_\eps$ in \eqref{eq:rpdeweak}. In view of the ellipticity of $A^\eps$ and \eqref{eq:DL2}, we have
\begin{equation}\label{eq:energy_bdd}
\begin{aligned}
& \lambda \left(\|\nabla \uepsp\|^2_{L^2(\Depsp)} + \|\nabla \uepsm\|^2_{L^2(\Depsm)}\right) + \frac{1}{\eps}\|\uepsp - \uepsm\|_{L^2(\Gameps)}^2
\le \|f\|_{L^2(D)}\|u_\eps\|_{L^2(D)}\\
\le \, &C\|f\|_{L^2(D)}\left( \|\nabla \uepsp\|_{L^2(\Depsp)} + \|\nabla \uepsm\|_{L^2(\Depsm)} + \eps \|\uepsp - \uepsm\|_{L^2(\Gameps)} \right)\\
\le \, &C\|f\|^2_{L^2(D)} + \frac{\lambda}{2} \left(\|\nabla \uepsp\|^2_{L^2(\Depsp)} + \|\nabla \uepsm\|^2_{L^2(\Depsm)}\right) + \eps \|u_\eps^+ - u_\eps^-\|^2_{L^2(\Gameps)}.
\end{aligned}
\end{equation}
In the last line above, we used the inequality that $ab \le \epsilon a^2 + \frac{1}{4\epsilon} b^2$. Since $\eps \le 1/\sqrt{2}$ is small, $\eps^{-1} - \eps \ge 1/(2\eps)$, and we have
\begin{equation*}
\frac{\lambda}{2} \left(\|\nabla \uepsp\|^2_{L^2(\Depsp)} + \|\nabla \uepsm\|^2_{L^2(\Depsm)}\right) + \frac{1}{2\eps} \|\uepsp - \uepsm\|_{L^2(\Gameps)} \le C\|f\|^2_{L^2(D)}.
\end{equation*} 
This yields \eqref{eq:GradientEst} and \eqref{eq:L2GammaEst}. In particular, $C$ depends on $\lambda$ but not on $\eps$ or $\omega$.
\end{proof}

Let $P^\eps_\omega$ denotes the extension operator of Proposition \ref{prop:ext_D}. Then we have the following estimates.

\begin{corollary}\label{cor:extlim}
Assume the same conditions of Proposition \ref{prop:energy}.
Let $P^\eps_\omega$ be the extension
operator. Then there exists a constant $C$, which is independent of $\eps$ and $\omega$,
such that
\begin{equation}
\|P^\eps_\omega \uepsp\|_{H^1(D)} \le C \|f\|_{L^2(D)}, \label{eq:BddExt}
\end{equation}
\begin{equation}
\|P^\eps_\omega \uepsp - \ueps \|_{L^2(D)} \le C \eps \|f\|_{L^2(D)}.
\end{equation}
\end{corollary}

\begin{proof} The first inequality follows immediately from \eqref{eq:GradientEst} and the Poincar\'e inequality \eqref{eq:poincare}. For the second inequality, in view of \eqref{eq:DmL2}, we have
\begin{eqnarray*}
\|P^\eps_\omega \uepsp - \ueps \|_{L^2(D)} &=& \|P^\eps_\omega \uepsp - \uepsm \|_{L^2(\Depsm)} 
\\
&\le &C\sqrt{\eps} \|P^\eps_\omega
\uepsp - \uepsm \|_{L^2(\Gameps)} + C\eps
\|\nabla(P^\eps_\omega \uepsp - \uepsm)
\|_{L^2(\Depsm)}.
\end{eqnarray*}
Thanks to \eqref{eq:L2GammaEst}, the first item on the right is bounded by $C\eps\|f\|_{L^2}$. For the second term, we have
\begin{equation*}
\eps \|\nabla(P^\eps_\omega \uepsp - \uepsm)\|_{L^2(\Depsm)} \le \eps \|\nabla P^\eps_\omega \uepsp\|_{L^2(D)} + \eps \|\nabla \uepsm\|_{L^2(\Depsm)},
\end{equation*}
and in view of \eqref{eq:L2GammaEst}, it is bounded by $C\eps \|f\|_{L^2(D)}$. 
\end{proof}

\section{The Auxiliary Problem}
\label{sec:auxil}

In this section we solve the auxiliary problem \eqref{eq:auxil}, which provides building blocks to carry out the oscillating test function methods of homogenization theory; see e.g. \cite{murat,Monsur,PV79}. We first explain the difficulty and outline the strategy.

We aim to find weak (distributional) solution $w_p(\cdot,\omega)$ of \eqref{eq:auxil} for a.e. $\omega \in \Omega$, such that $\nabla w_p (\Phi(\cdot,\cdot),\cdot)$, as a random field, is stationary. The natural functional space for the solutions to \eqref{eq:auxil}, i.e. $w_p(y,\omega)$, where $y \in \R^d$ and $\omega \in \Omega$, is
\begin{equation}\label{e.Hcal}
\Hcal := \{w=\widetilde{w}\circ \Phi^{-1}~|~ \widetilde{w} \in \widetilde{\Hcal}\} 
\quad \text{where} \quad
\widetilde{\Hcal} := L^2(\Omega, H^1_{\rm loc}(\Rdp) \times H^1_{\rm loc}(\Rdm)).
\end{equation}
As mentioned in the introduction and due to the lack of compactness, however, there is no Poincar\'e inequality for the auxiliary problem, and it is not clear at all how to solve it for each realization. We follow the standard procedures of Papanicolaou and Varadhan \cite{PV79} and of Kozlov \cite{Kozlov}, add a small zero-order term to regularize the problem, and wish to obtain a meaningful limit by passing the regularization is sent to zero. This leads us to investigate the following regularized problem.
\begin{equation}
\left\{
\begin{aligned}
&-\nabla \cdot\left( A (y,\omega) \nabla \left[w_{p,\delta}^\pm(y,\omega) + p\cdot y\right] \right) + \delta w_{p,\delta}^\pm= 0, \quad &\text{ in } \Phi(\Rdpm,\omega),\\
&\frac{\partial}{\partial \nu_A} w_{p,\delta}^+(y,\omega) = \frac{\partial}{\partial \nu_A} w_{p,\delta}^-(y,\omega), \quad &\text{ on }  \Phi(\Gamd,\omega),\\
&w_{p,\delta}^+(y,\omega) - w_{p,\delta}^-(y,\omega) = \frac{\partial}{\partial \nu_A} w_{p,\delta}^+(y,\omega) + \nu_y \cdot Ap, \quad &\text{ on }  \Phi(\Gamd,\omega),\\
&w_{p,\delta}^\pm(y) = \widetilde{w}_{p,\delta}^\pm \circ \Psi(y), \quad \widetilde{w}_{p,\delta}^\pm(z,\omega) \text{ is stationary}. &
\end{aligned}
\right.
\label{eq:regul}
\end{equation}

If there were no complicated structure $\Phi(\Gamd,\omega)$, the above would reduce to an elliptic problem over the whole space but with an absorptive potential. The Green's function, for each realization, decays exponentially fast and the solution is standard. Even in the setting of this paper, with the regularization, we still expect that $w_{p,\delta}$ can be controlled at infinity, and that $w_{p,\delta}(\cdot,\cdot)$ belongs to a subspace $\Hcal_S$ of $\Hcal$:
\begin{equation}\label{e.HcalS}
\Hcal_S := \{w=\widetilde{w}\circ \Phi^{-1}~|~ \widetilde{w} \in \widetilde{\Hcal}_S\}, \quad \text{where} \quad
\widetilde{\Hcal}_S := L^2(\Omega, H^1_{\rm loc}(\Rdp) \times H^1_{\rm loc}(\Rdm)),
\end{equation}
which contains stationary functions. If the notion of stationarity is with respect to all $\R^d$-translations, there is a natural correspondence between the differential operator $\nabla$ in $\R^d$ and the infinitesimal generator $D$ of the translation group $\{\tau_x\}_{x\in \R^d}$ on $\Omega$, and any stationary field $\widetilde{w}(\cdot,\omega) \in H^1_{\rm loc}(\R^d)$ is of the form $\widetilde{w}(0,\tau_\cdot \omega)$, and hence can be identified the random variable $\hat{w}(\omega) = \widetilde{w}(0,\omega)$ that satisfies $\hat{w}, D\hat{w} \in L^2(\Omega)$. Were this the case and were the structure $\Phi(\Gamd,\omega)$ absent, \eqref{eq:regul} would be lifted to the space $H^1(\Omega) = \{\hat{w} \in L^2(\Omega) ~|~ D\hat{w} \in L^2(\Omega)\}$. Then $\hat{w}$ would be found for the equation in the probability space and $\hat{w}(\tau_x \omega)$ would solve the problem in the physical space. Due to the $\Z^d$-stationarity and the presence of $\Phi(\Gamd,\omega)$, the method above does not work and we have to solve \eqref{eq:regul} in the physical space. 

Our first natural idea is to follow the approach of \cite{BLBL06} and take advantage of the fact that $\Hcal_S$, equipped with the inner product
\begin{equation}
\langle u, v \rangle_{\Hcal_S} := \E \left[\int_{Y^+} \nabla \widetilde{u}^+\cdot \nabla \widetilde{v}^+ (z,\cdot)dz + \int_{Y^-} \nabla \widetilde{u}^+\cdot \nabla \widetilde{v}^- (z,\cdot)dz +\int_Y \widetilde{u}\widetilde{v}(z,\cdot) dz\right],
\end{equation}
is a Hilbert space, where the relation $\widetilde{u}(z,\cdot) = u\circ\Phi(z,\cdot)$ is used. Let $\Acal_\delta: \Hcal_S \times \Hcal_S \to \R$ be the bilinear form
\begin{equation}
\begin{aligned}
\Acal_\delta(u,v) = \E\left[\int_{\Phi(Y^+,\cdot)} A\nabla u^+\cdot \nabla v^+ (y,\cdot)dy + \int_{\Phi(Y^-,\cdot)} A\nabla u^-\cdot \nabla v^- (y,\cdot) dy \right.\\
\left. + \int_{\Phi(\Gamma_0,\cdot)} (u^+-u^-)(v^+-v^-)(y,\cdot)d\sigma(y) + \delta \int_{\Phi(Y,\cdot)} uv(y,\cdot) dy \right].
\end{aligned}
\end{equation}
Let $F_p: \Hcal_S \to \R$ be the linear form
\begin{equation}
F_p(v) = - \E \left[\int_{\Phi(Y^+,\cdot)} A p \cdot \nabla v^+(y,\cdot) dy + \int_{\Phi(Y^-,\cdot)} A p \cdot \nabla v^-(y,\cdot) dy \right].
\end{equation}
It is easy to check that, due to \eqref{eq:ellip} and the assumptions (A1)(A2)(A3), for each fixed $\delta>0$, $\Acal_\delta$ is bi-continuous and coercive on $\Hcal_S$, and $F_p$ is continuous on $\Hcal_S$. By the Lax-Milgram theorem, there exists a unique $w_{p,\delta} = (w_{p,\delta}^+,w^-_{p,\delta}) \in \Hcal_S$ such that
\begin{equation}
\Acal_\delta (w_{p,\delta}, v) = F_p(v), \quad \text{for all } v \in \Hcal_S.
\label{eq:lax_mil}
\end{equation}
Essentially, this is like lifting \eqref{eq:regul} to the space $\Hcal_S$. However, it is not clear at this stage that $w_{p,\delta}(\cdot,\omega)$, for each $\omega$, solves \eqref{eq:regul} in the distributional sense. This is because of the lack of correspondence between $\nabla$ and $D$ and the lack of Weyl's decomposition type of result on $\Hcal_S$. Indeed, for $w_{p,\delta}(\cdot,\omega)$ to be a distributional solution, we need to verify that: for all $\phi(\cdot) \in H^1_{\rm loc}(\Phi(\Rdp,\omega)) \times H^1_{\rm loc}(\Phi(\Rdm,\omega))$ with support $K\subset\subset \R^d$, 
\begin{equation}
\begin{aligned}
\int_{K\cap \Phi(\Rdp,\omega)} A(\nabla w_{p,\delta}^+ + p)\cdot \nabla \phi^+ dy \ + \ \int_{K\cap \Phi(\Rdm,\omega)}  A(\nabla w_{p,\delta}^- + p)\cdot \nabla \phi^- dy \\
\ + \ \delta \int_{K} w_{p,\delta} \phi \ dy \ + \  \int_{K \cap \Phi(\Gamd,\omega)} (w^+_{p,\delta} - w_{p,\delta}^-)(\phi^+ -\phi^-) d\sigma(y) = 0.
\end{aligned}
\label{eq:weak_pw_reg}
\end{equation} 
The solution given by \eqref{eq:lax_mil}, however, only take test functions $\varphi \in \Hcal_S$ and the integrals are over $\Omega \times \R^d$. So we need a new method. We remark that our method, at the end, establishes the (non-trivial) fact that the solution given by \eqref{eq:lax_mil} agrees with the distributional solution, for a.e. $\omega \in \Omega$; see Remark \ref{r.laxmil} below. This is because the stationary structure of the problem forces the two notions of solutions to agree, but this fact needs to be proved more carefully, unlike in the setting of $\R^d$-stationary with simple geometry where this equivalence is clear.

We explain our strategy to solve the regularized problem \eqref{eq:regul}. The convenient functional space for this problem is $H^1_{\rm uloc}$, locally uniform $H^1$ spaces. More precisely, we recall that $Y_k, Y^+_k, Y^-_k$, $k\in \Z^d$, are translated cubes (cells inside the cubes, etc.) and define
\begin{equation}
\begin{aligned}
H^1_{\rm uloc}(\R^d) &:= \{\widetilde{v} \in H^1_{\rm loc}(\R^d) ~|~ \sup_{k\in \Z^d} \|\widetilde{v}\|_{H^1(Y_k)} < \infty\},\\
H^1_{\rm uloc}(\Rdp) &:= \{\widetilde{v} \in H^1_{\rm loc}(\Rdp) ~|~ \sup_{k\in \Z^d} \|\widetilde{v}\|_{H^{1}(Y^+_k)} < \infty\},\\ 
H^1_{\rm uloc}(\Rdm) &:= \{\widetilde{v} \in H^1_{\rm loc}(\Rdm) ~|~ \sup_{k\in \Z^d} \|\widetilde{v}\|_{H^{1}(Y^-_k)} < \infty\}.\end{aligned}
\end{equation}
Locally uniformly Sobolev spaces are sometimes called Sobolev-Kato spaces; see \cite{Kato,Lady,cholewa}. They are natural for problems posed on unbounded domains without decaying condition at infinity; functions in this case have infinite energy over the whole space. The metric in $H^1_{\rm uloc}$ requires uniform controls over bounded sets; and weak compactness, say in $L^2_{\rm loc}$, can be obtained by applying Rellich's theorem on a sequence of bounded sets that exhaust the whole space. 

We expect that for each $\omega\in \Omega$, the solution $w_{p,\delta}(\cdot,\omega)$ is of the form $\widetilde{w}_{p,\delta}(\Psi(\cdot,\omega),\omega)$ and $\widetilde{w}_{p,\delta}(\cdot,\omega)$ belongs to $H^1_{\rm uloc}(\Rdp) \times H^1_{\rm uloc}(\Rdm)$. Indeed, with the regularization, we expect to be able to control $w_{p,\delta}$ at infinity, and since the position of the origin does not play a role in the problem \eqref{eq:regul}, we expect the control to be uniform in the sense shown in the definitions of the above spaces. Our strategy to solve \eqref{eq:regul}, for each $\omega \in \Omega$, is to truncate the problem on an increasing sequence of sets $\Phi(Q_n,\omega)$ that exhaust $\R^d$, prove that the sequence of solutions to the truncated problems (extended by zero outside $\Phi(Q_n,\omega)$) are uniformly bounded in $H^1_{\rm uloc}(\Rdp) \times H^1_{\rm uloc}(\Rdm)$, and show that a subsequence converges to the solution of \eqref{eq:regul}. More precisely, we prove

\begin{lemma}\label{lem:regul}
For each fixed $\delta > 0$ and $p \in \R^d$, there exists $w_{p,\delta} \in \Hcal_S$ such that for each $\omega \in \Omega$, $w_{p,\delta}(\cdot,\omega) = \widetilde{w}_{p,\delta}(\Psi(\cdot,\omega),\omega)$, $\widetilde{w}_{p,\delta} \in H^1_{\rm uloc}(\Rdp) \times H^1_{\rm uloc}(\Rdm)$, and $w_{p,\delta}(\cdot,\omega)$ solves the regularized problem \eqref{eq:regul} in the distributional sense. In addition, there exists a constant $C$ which depends only on $d, M, \nu,\Gamma_0$ and $p$, such that
\begin{equation}
\begin{aligned}
& \E \int_{\Phi(Y^+,\cdot)} \lvert\nabla w_{p,\delta}^+(y,\cdot)\rvert^2 dy \ +\ \E \int_{\Phi(Y^-,\cdot)} \lvert\nabla w_{p,\delta}^-(y,\cdot)\rvert^2 dy \le C,\\
& \E \int_{\Phi(\Gamma_0,\cdot)} (w_{p,\delta}^+-w_{p,\delta}^-)^2(y,\cdot) d\sigma(y) \le C, \quad \E\int_{\Phi(Y,\cdot)} \delta \lvert w_{p,\delta}(y,\cdot)\rvert^2 dy \le C.
\end{aligned}
\label{eq:regest}
\end{equation}
\end{lemma}

We note that the uniform in $\delta$ bounds on $\nabla w_{p,\delta}$ above are after taking expectation $\E$. In particular, for each $\omega \in \Omega$, we do not have uniform in $\delta$ bounds of the $H^1_{\rm uloc}$ bounds on $w_{p,\delta}(\cdot,\omega)$.

\begin{proof}{\it Step 1: Construction of weak solution for fixed $\omega \in \Omega$.} In this step,  a realization $\omega$ is fixed and, for notational simplification, we omit the dependence on $\omega$ below. Constants $C$ below are made sure independent of $\omega$. Our goal is to find a unique $w_{p,\delta} \in H^1_{\rm uloc}(\Phi(\Rdp)) \times H^1_{\rm uloc}(\Phi(\Rdm))$ such that \eqref{eq:weak_pw_reg} holds.

{\em Existence.} For each positive integer $n \in \N$, let $Q_n = (-n,n)^d$ be the cube with side length $2n$ and $\Phi(Q_n,\omega)$  be the deformed cube. Consider the truncated problem with Dirichlet boundary conditions on $\partial \Phi(Q_n)$:
\begin{equation}
\left\{
\begin{aligned}
&-\nabla \cdot\left( A (y,\omega) \nabla \left[w_{p,\delta,n}^\pm(y,\omega) + p\cdot y\right] \right) + \delta w_{p,\delta,n}^\pm= 0, &\quad &\text{ in }  \Phi(Q_n\cap \Rdpm),\\
&\frac{\partial}{\partial \nu_A} w_{p,\delta,n}^+(y,\omega) = \frac{\partial}{\partial \nu_A} w_{p,\delta,n}^-(y,\omega), &\quad &\text{ in }  \Phi(Q_n\cap \Gamd),\\
&w_{p,\delta,n}^+(y,\omega) - w_{p,\delta,n}^-(y,\omega) = \frac{\partial}{\partial \nu_A} w_{p,\delta,n}^+(y,\omega) + \nu_y \cdot Ap, &\quad &\text{ on }  \Phi(Q_n\cap\Gamd),\\
&w_{p,\delta,k}^+(y,\omega) = 0, &\quad &\text{ on }  \Phi(\partial Q_n).
\end{aligned}
\right.
\label{eq:piece}
\end{equation}
This is a problem on bounded domain and there is the absorption term. The existence and uniqueness of $w_{p,\delta,n} = (w^+_{p,\delta,n},w^-_{p,\delta,n})$ in the space $\Wcal_n := \{w = \widetilde{w}\circ \Phi^{-1} ~|~ \widetilde{w} \in \widetilde{\Wcal}_n\}$, where
\begin{equation*}
\widetilde{\Wcal}_n := \{\widetilde{w} \in H^1(Q_{n} \cap \Rdp) \times H^1(Q_{n}\cap \Rdm) ~|~ \widetilde{w} = 0 \text{ on } \partial Q_n\},
\end{equation*}
follows directly by applying the Lax-Milgram theorem in the above Hilbert space. The solution satisfies, for any $\phi \in \Wcal_n$,
\begin{equation}
\begin{aligned}
& \int_{\Phi(Q_n\cap \Rdp)} A(y)\nabla w^+_{p,\delta,n}\cdot \nabla \phi^+ dy  + \int_{\Phi(Q_n\cap \Rdm)} A(y)\nabla w^-_{p,\delta,n}\cdot \nabla \phi^- dy  \\
& \quad + \delta \int_{\Phi(Q_n)}  w_{p,\delta,n} \phi\ dy + \int_{\Phi(Q_n\cap \Gamd)} (w^+_{p,\delta,n}-w^-_{p,\delta,n})(\phi^+ - \phi^-) d\sigma(y)\\
= &\quad -\int_{\Phi(Q_n\cap \Rdp)} A(y)p \cdot \nabla \phi^+ dy  - \int_{\Phi(Q_n\cap \Rdm)} A(y)p\cdot \nabla \phi^- dy.
\end{aligned}
\label{eq:weak_piece}
\end{equation}
Extending functions in $\widetilde{\Wcal}_n$ by zero outside $Q_n$, we find $\widetilde{\Wcal}_n \subset H^1_{\rm uloc}(\Rdp) \times H^1_{\rm uloc}(\Rdm)$ for all $n$. Hence, we have obtained a sequence $\{w_{p,\delta,n}\}_{n=1}^\infty$ that belong to $\Wcal : = H^1_{\rm uloc}  (\Phi(\Rdp))\times H^1_{\rm uloc}(\Phi(\Rdm))$ and each of them solves \eqref{eq:piece}. 

Next we establish uniform in $n$ bounds for $w_{p,\delta,n}$. In terms of $\widetilde{w}_{p,\delta,n} = w_{p,\delta,n}\circ \Phi$, we will show that for some $C = C(\lambda,\Lambda,\mu,M,d,\delta)$, independent of $n$, the quantities
\begin{equation}\label{eq:E-def}
\begin{aligned}
E_k  = &\int_{Q_k \cap \Rdp} |\nabla \widetilde{w}^+_{p,\delta,n}|^2 dx + \int_{Q_k \cap \Rdm} |\nabla \widetilde{w}^-_{p,\delta,n}|^2 dx + \delta \int_{Q_k} |\widetilde{w}_{p,\delta,n}|^2 dx \\
& \quad +  \int_{Q_k \cap \Gamd} (\widetilde{w}^+_{p,\delta,n} - \widetilde{w}^-_{p,\delta,n})^2 d\sigma(x)
\end{aligned}
\end{equation}
satisfy the estimate $E_k \le Ck^d$ for all $k \le n$.

For each $k \in \N$, let $Q_{k,k+1}$ be a short-hand notation for $Q_{k+1}\setminus Q_k$. We can find a smooth cutoff function $\chi_k = \widetilde{\chi}_k\circ \Phi^{-1}$ where $0 \le \widetilde{\chi}_k\le 1$ is a cutoff function supported on $Q_{k+1}$ and equals one in $Q_k$. These cutoff functions can be chosen such that $\|\nabla \widetilde{\chi}_k\|_{L^\infty} \lesssim 1$. We note also that $\nabla \widetilde{\chi}_k$ is supported in $Q_{k,k+1}$. Set $\widetilde{\phi} = \widetilde{\chi}_k \widetilde{w}_{p,\delta,n}$ and take $\phi = \widetilde{\phi}\circ \Phi^{-1}$ in the weak formulation \eqref{eq:weak_piece}; we get
\begin{equation}
\begin{aligned}
& \int_{\Phi(Q_{k+1} \cap \Rdp)} \chi_k A\nabla w^+_{p,\delta,n}\cdot \nabla w^+_{p,\delta,n} dy + 
\int_{\Phi(Q_{k+1} \cap \Rdm)} \chi_k A\nabla w^-_{p,\delta,n}\cdot \nabla w^-_{p,\delta,n} dy\\
&+ \int_{\Phi(Q_{k+1} \cap \Gamd)} \chi_k(w^+_{p,\delta,n} - w^-_{p,\delta,n})^2 d\sigma(y) + \delta \int_{\Phi(Q_{k+1})} \chi_k |w_{p,\delta,n}|^2 dy\\
&+  \int_{\Phi(Q_{k,k+1}\cap \Rdp)} w^+_{p,\delta,n}A\nabla w^+_{p,\delta,n} \cdot \nabla \chi_k dy + \int_{\Phi(Q_{k,k+1}\cap \Rdm)} w^-_{p,\delta,n}A\nabla w^-_{p,\delta,n} \cdot \nabla \chi_k dy \quad \\
\label{weak-c0}
= &\quad - \int_{\Phi(Q_{k+1} \cap \Rdp)} \chi_k Ap\cdot \nabla w^+_{p,\delta,n} dy - 
\int_{\Phi(Q_{k+1} \cap \Rdm)} \chi_k Ap\cdot \nabla w^-_{p,\delta,n} dy\\
&-  \int_{\Phi(Q_{k,k+1}\cap \Rdp)} w^+_{p,\delta,n}Ap \cdot \nabla \chi_k dy - \int_{\Phi(Q_{k,k+1}\cap \Rdm)} w^-_{p,\delta,n}Ap\cdot \nabla \chi_k dy.
\end{aligned}
\end{equation}
The first four terms on the left are good ones because they are positive. Let us denote the absolute value of the other terms by $I_5, I_6, I_7,I_8,I_9$ and $I_{10}$, according to their order of appearance in the equation above. These terms can be bounded by using H\"older inequality and the inequality $ab \le \epsilon a^2 + \frac{1}{4\epsilon} b^2$. In particular, $I_7,I_8$ involve integrals over $\Phi(Q_{k+1})$, and we have
\begin{equation}\label{weak-c1}
I_7 \le \frac{\lambda}{2}\int_{\Phi(Q_{k+1}\cap \Rdp)} |\nabla w^+_{p,\delta,n}|^2 dy + C(k+1)^d,
\end{equation}
where $C=C(\lambda,\Lambda,\mu,M,p)$ and $(k+1)^d$ is the order of the volume of $\Phi(Q_{k+1})$; $I_8$ shares the same estimate with $w^+_{p,\delta,n}$ replaced by $w^-_{p,\delta,n}$. 

The terms $I_5,I_6,I_{9},I_{10}$ involve integrals over $\Phi(Q_{k,k+1})$ which is the space between $\Phi(Q_k)$ and $\Phi(Q_{k+1})$. Furthermore, they involve integrals of $w^\pm_{p,\delta,n}$, which appears as $\delta|w^\pm_{p,\delta,n}|^2$ in the definition of $E_k$. Therefore, we control them as follows:
\begin{equation*}
\begin{aligned}
I_5 &\le \frac{\delta}{2} \int_{\Phi(Q_{k,k+1}\cap \Rdp)} |w^+_{p,\delta,n}|^2 dy + \frac{C}{\delta} \int_{\Phi(Q_{k,k+1}\cap \Rdp)} |\nabla w^+_{p,\delta,n}|^2 dy,\\
I_9 &\le \frac{\delta}{2} \int_{\Phi(Q_{k,k+1})} |w^+_{p,\delta,n}|^2 dy + \frac{C}{\delta} (k+1)^d,
\end{aligned}
\end{equation*}
and $I_6, I_{10}$ have similar bounds, with $w^+_{p,\delta,n}$ replaced by $w^-_{p,\delta,n}$. In \eqref{weak-c1}, we further divide the integral over $\Phi(Q_{k+1})$ into two pieces: one over $\Phi(Q_k)$ and the other over $\Phi(Q_{k,k+1})$. We have deliberately made the coefficient in front of the integral in \eqref{weak-c1} less than the corresponding ones on the left hand side of \eqref{weak-c0}. Hence, after some cancellation we have
\begin{equation*}
\begin{aligned}
&\frac{1}{2}\left[\int_{\Phi(Q_k\cap \Rdp)} A\nabla w^+_{p,\delta,n}\cdot \nabla w^+_{p,\delta,n} dy \ + \  \int_{\Phi(Q_k\cap \Rdm)} A\nabla w^-_{p,\delta,n}\cdot \nabla w^-_{p,\delta,n} dy \right.\\
 & \qquad \qquad\left. + \ \delta \int_{\Phi(Q_k)} |w_{p,\delta,n}|^2 dy \ + \ \int_{\Phi(Q_k\cap \Gamd)} (w^+_{p,\delta,n} - w^-_{p,\delta,n})^2 d\sigma(y) \right]\\
\le\quad & \frac{C}{\delta} \left[\int_{\Phi(Q_{k,k+1}\cap \Rdp)} |\nabla w^+_{p,\delta,n}|^2 dy \ + \  \int_{\Phi(Q_{k,k+1}\cap \Rdm)} |\nabla w^-_{p,\delta,n}|^2 dy + (k+1)^d \right] \\
& \qquad \qquad +\ \delta \int_{\Phi(Q_{k,k+1})} |w_{p,\delta,n}|^2 dy.
\end{aligned}
\end{equation*}
After changing variable $y$ to $\Phi(x)$ in the integrals above, we find that this inequality shows
\begin{equation}
E_k \le C_\delta (E_{k+1}-E_k + (k+1)^d), \quad \text{ for all } k \le n.
\label{eq:saintV}
\end{equation}
The constant $C_\delta$ depends on $\lambda, \Lambda, \mu, M, p, d,\delta$ and in particular it is of order $\delta^{-1}$, but it is independent of $n$ and $k$. 

We observe that $E_n \le C n^d$ for some $C(\lambda,\Lambda,\mu,M,p,d)$, which follows from \eqref{eq:weak_piece}. By a standard backward induction (see e.g. the proof of Proposition 10 in \cite{GV_Masmoudi}), there exists another positive integer $C'_\delta$ such that
\begin{equation}
E_k \le C'_\delta k^d, \quad \text{ for all } k \le n.
\label{eq:Ek_bd}
\end{equation}
In particular, we have $E_1 \le C'_\delta$. By definition, $\widetilde{w}^\pm_{p,\delta,n}$ vanishes at $\partial Q_n$. Hence, if we extend it by zero outside $Q_n$, then $\widetilde{w}^\pm_{p,\delta,n} \in H^1_{\rm uloc}(\Rdp) \times H^1_{\rm uloc}(\Rdm)$. The bound in \eqref{eq:Ek_bd} with $k=1$ shows
\begin{equation*}
\|\nabla \widetilde{w}^+_{p,\delta,n}\|^2_{L^2(Y^+)} + \|\nabla \widetilde{w}^-_{p,\delta,n}\|^2_{L^2(Y^-)} + \|\widetilde{w}^+_{p,\delta,n} - \widetilde{w}^-_{p,\delta,n}\|^2_{L^2(\Gamma_0)} \le C'_\delta.
\end{equation*}
Examining the estimates on the items in \eqref{weak-c0}, and the estimate for $E_n$, we observe that they are all translation invariant. In particular, the bounds on $E_n$ does not change because if we restrict $\widetilde{w}_{p,\delta,n}$ to the shifted cube $k + Q_{2n}$, then its norm will be smaller since outside $Q_{2n}$ the function is extended by zero. Therefore, we get
\begin{equation*}
\sup_{k \in \Z^d} \left(\|\widetilde{w}^+_{p,\delta,n}\|^2_{H^1(k+Y^+)} + \|\widetilde{w}^-_{p,\delta,n}\|^2_{H^1(k+Y^-)} + \|\widetilde{w}^+_{p,\delta,n} - \widetilde{w}^+_{p,\delta,n}\|^2_{L^2(k+\Gamma_0)}\right) \le C''_\delta.
\end{equation*}
This is an uniform bound for the $H^1_{\rm uloc}(\Rdp)\times H^1_{\rm uloc}(\Rdm)$ norm of $\{w_{p,\delta,n}\}$. As a result, we extract a subsequence of $\widetilde{w}_{p,\delta,n_k}$ that converges weakly to some $\widetilde{w}_{p,\delta} \in H^1_{\rm uloc}(\Rdp)\times H^1_{\rm uloc}(\Rdm)$. The function $w_{p,\delta} = \widetilde{w}_{p,\delta}\circ \Phi^{-1}$ then satisfies \eqref{eq:weak_pw_reg}, i.e. solving \eqref{eq:regul} in the distributional sense. 

{\em Uniqueness.} Given any two solutions $w^{(1)}_{p,\delta}$ and $w^{(2)}_{p,\delta}$ that satisfy \eqref{eq:weak_pw_reg}. Let $v_{p,\delta}$ denotes their difference. Then this function satisfies \eqref{eq:weak_piece} with $w_{p,\delta,n}$ replaced by $v_{p,\delta}$ and $p = 0$. The analysis that follows \eqref{eq:weak_piece} can be repeated, and in particular, the $(k+1)^d$ term in \eqref{eq:saintV} disappears and this yields
\begin{equation*}
E_k\le C_\delta(E_{k+1}-E_k), \quad \forall k\le n.
\end{equation*}
Here, $E_k$ is defined as in \eqref{eq:E-def} with $\widetilde{w}_{p,\delta,n}$ replaced by $\widetilde{v}_{p,\delta}$. The above inequality shows that $E_k \le \eta_\delta E_{k+1}$ for some $0<\eta_\delta<1$ for all $n$ and all $k\le n$. By assumption, $\widetilde{w}^{(j)}_{p,\delta,n}$, $j=1,2$ and $\widetilde{v}_{p,\delta}$ are in $\widetilde{\Wcal} = H^1_{\rm uloc}(\Rdp) \times H^1_{\rm uloc}(\Rdm)$, we have $E_n \le Cn^d\|\widetilde{v}_{p,\delta}\|_{\widetilde{\Wcal}}$. By standard backward induction, we get
\begin{equation}
E_1 \le \eta_\delta^{n-1}E_n \le C\eta_\delta^{n-1} n^d \|\widetilde{v}_{p,\delta}\|_{\widetilde{\Wcal}}.
\end{equation}
Let $n\to \infty$, we get $E_1 = 0$, which implies that $\widetilde{v}_{p,\delta} \equiv 0$ in $Y$. By translation invariance of the above argument, $\widetilde{v}_{p,\delta} = 0$ over the whole space. This proves the uniqueness of the weak solution satisfying \eqref{eq:weak_pw_reg}.

{\itshape Step 2: Stationarity.} Let $w_{p,\delta}$ and $\widetilde{w}_{p,\delta}$ be as defined above. We observe that $\widetilde{w}_{p,\delta}(\cdot+k,\omega) = \widetilde{w}_{p,\delta}(\cdot,\tau_k \omega)$ for all $k \in \Z^d$ and $\omega \in \Omega$. Indeed, due to the stationarity of the parameters in \eqref{eq:regul} and the domain on which the problem is posed, we check directly that $w_{p,\delta}(\cdot+k,\omega)$ is a weak solution to \eqref{eq:regul}, with realization $\tau_k \omega$, in the sense of \eqref{eq:weak_pw_reg}. On the other hand, due to uniqueness we just proved, $w_{p,\delta}(\cdot+k,\omega)$ has to agree with $w_{p,\delta}(\cdot,\tau_k \omega)$. This shows that $\widetilde{w}_{p,\delta}$ is stationary, i.e belonging to $\widetilde{{\Hcal}}_S$.

{\itshape Step 3: Uniform bounds.} To prove \eqref{eq:regest}, we multiply $w^\pm_{p,\delta}(\cdot,\omega)$ on \eqref{eq:regul} and integrate over $\Phi(Y,\omega)$ to get
\begin{equation*}
\begin{aligned}
\int_{\Phi(Y^+)} A\left(p+\nabla w_{p,\delta}^+\right) \cdot \nabla w^+_{p,\delta} dy + \int_{\Phi(Y^-)} A\left(p+\nabla w_{p,\delta}^-\right) \cdot \nabla w^-_{p,\delta} dy 
+ \delta \int_{\Phi(Y)} (w_{p,\delta})^2 dy\\ + \int_{\Phi(\Gamma_0)} (w^+_{p,\delta}- w^-_{p,\delta})^2 d\sigma(y) = \int_{\Phi(\partial Y)} w^+_{p,\delta} \nu_y \cdot A\left(p+\nabla w^+_{p,\delta}\right) d\sigma(y).
\end{aligned}
\end{equation*}
Note that the right hand side is nonzero because $w^+_{p,\delta}$ does not vanish on $\Phi(\partial Y)$. After a change of variable and in light of the formulas \eqref{eq:geo-f1} and \eqref{eq:geo-f2}, the last term can be written as
\begin{equation*}
\int_{\partial Y} \widetilde{w}^+_{p,\delta} \left(D\Phi(x)^{-1}\right)^{t} \nu_x \cdot \widetilde{A}\left(p + \left(D\Phi(x)^{-1}\right)^{t} \nabla \widetilde{w}^+_{p,\delta}\right) \det (D\Phi(x)) \ d\widetilde{\sigma} (x).
\end{equation*}
The integration of this term over $\Omega$ vanishes because the functions in the integrand are stationary except that $\nu_x$ has opposite signs when evaluated at a pair of opposite sides of $Y$. Therefore, after taking expectation, we get $\Acal_\delta (w_{p,\delta}, w_{p,\delta}) = F_p(w_{p,\delta})$ in the sense of \eqref{eq:lax_mil}. This yields the bounds \eqref{eq:regest} and completes the proof of the lemma.
\end{proof}

\begin{rem}\label{r.laxmil} The argument in step 3, in fact, proves that \eqref{eq:lax_mil} holds for $w_{p,\delta}$ and for all $v \in \Hcal_S$. As a result, in view of the uniqueness of the solution to \eqref{eq:lax_mil}, $w_{p,\delta}$ found above agrees (in the space $\Hcal_S$) with the seemingly weaker solution defined by the Lax-Milgram approach on $\Hcal_S$.
\end{rem}
\subsection{Proof of Theorem \ref{thm:auxil}}

We first explain our strategy. The idea is clear and amounts to sending the regularization parameter $\delta \to 0$ in \eqref{eq:regul}, and showing that the limit solves \eqref{eq:auxil}. Since the available (uniform in $\delta$) bounds \eqref{eq:regest} are about the expectations, it is natural to define the following notion of solution to \eqref{eq:auxil}.

We say that $w_p  \in \Hcal = L^2(\Omega, H^1_{\rm loc}(\Phi(\Rdp)) \times H^1_{\rm loc}(\Phi(\Rdm)))$, is an {\it annealed weak solution} to \eqref{eq:auxil} if for any $\varphi$ in $\Hcal$ with support $K \subset\subset \R^d$, it holds that
\begin{equation}
\begin{aligned}
& \E \int_{K\cap \Phi(\Rdp,\omega)} A(\nabla w_p^+ + p)\cdot \nabla \varphi^+ dy \ + \ \E \int_{K\cap \Phi(\Rdm,\omega)}  A(\nabla w_p^- + p)\cdot \nabla \varphi^- dy \\
+ & \quad \E \int_{K \cap \Phi(\Gamd,\omega)} (w^+_p - w^-_p)(\varphi^+ -\varphi^-) d\sigma(y) = 0.
\end{aligned}
\label{eq:weakaux}
\end{equation}

\begin{proposition}\label{prop:weak}
Suppose $\Hcal$ is separable and $w_p \in \Hcal$ is an annealed weak solution to \eqref{eq:auxil}. Then there exists $\Omega_0 \subset \Omega$ with $\Pb(\Omega_0) = 1$, such that for all $\omega \in \Omega_0$, for all $\widetilde{\phi} \in H^1_{\rm loc}(\Rdp) \times H^1_{\rm loc}(\Rdm)$, $\widetilde{\phi}$ supported in $K$ and $\phi(\cdot,\omega) = \widetilde{\phi} \circ \Phi^{-1}(\cdot,\omega)$,
\begin{equation}
\begin{aligned}
& \int_{\Phi(K\cap \Rdp,\omega)} A(\nabla w_p^+ + p)\cdot \nabla \phi^+ dy + \int_{\Phi(K\cap \Rdm,\omega)}  A(\nabla w_p^- + p)\cdot \nabla \phi^- dy \\
+ & \quad \int_{\Phi(K \cap \Gamd,\omega)} (w^+_p - w^-_p)(\phi^+ -\phi^-) d\sigma(y) = 0.
\end{aligned}
\label{eq:weakaux_pw}
\end{equation}
\end{proposition}

\begin{rem} The above proposition says that an annealed weak solution is also a distributional solution a.e. in $\Omega$. This is because the test functions in \eqref{eq:weakaux_pw} are rich enough. The proof of this proposition is in Appendix \ref{sec:equiv_weak}. We note that the separability of $\Hcal$ is true because we assumed that the probability space is countably generated. 
\end{rem}

\begin{rem}
We remark that the same conclusion holds for the regularized problem \eqref{eq:regul} also. In particular, using test function $\phi = \widetilde{\phi}\circ \Phi^{-1}$, with $\widetilde{\phi}$ of the form $\varphi(x)\psi(\omega)$ and $\varphi$ supported in $K \subset \subset \R^d$, and taking expectation in \eqref{eq:weak_pw_reg}, we verify that
\begin{equation}
\begin{aligned}
& \E \int_{\Phi(K\cap \Rdp,\omega)} A(\nabla w_{p,\delta}^+ + p)\cdot \nabla \phi^+ dy + \E \int_{\Phi(K\cap \Rdm,\omega)}  A(\nabla w_{p,\delta}^- + p)\cdot \nabla \phi^- dy \\
+ & \quad \E\ \delta\int_{\Phi(K,\omega)} w_{p,\delta} \phi\ dy + \E \int_{\Phi(K \cap \Gamd,\omega)} (w^+_{p,\delta} - w^-_{p,\delta})(\phi^+ -\phi^-) d\sigma(y) = 0.
\end{aligned}
\label{eq:weakregaux}
\end{equation}
Since functions of the form $\varphi(x)\psi(\omega)$ is dense in $\widetilde{\Hcal}$, the above equality still holds for test functions $\phi \in \Hcal$. As a result, $w_{p,\delta}$ obtained in Lemma \ref{lem:regul} is an annealed solution of \eqref{eq:regul} as well. We emphasize again, one cannot get this result directly from the Lax-Milgram approach in \eqref{eq:lax_mil}.
\end{rem}

Now we prove Theorem \ref{thm:auxil}. In view of Proposition \ref{prop:weak}, it suffices to find annealed solution $w_p \in \Hcal$. 

\begin{proof}[Proof of Theorem \ref{thm:auxil}] {\itshape Existence of annealed solution for each $p$}. Fix any $p \in \R^d$. Let $\{w_{p,\delta} \in \Hcal_S\}$ be as obtained in the previous section and let $P_\omega$ be the extension operator of Proposition \ref{prop:ext_Ro}. Then $\{w^{\rm ext}_{p,\delta} := P_\omega w^+_{p,\delta} ~|~ \delta > 0\}$ is a family of functions in $L^2(\Omega,H^1_{\rm loc}(\R^d))\}$. Set $\widetilde{w}^{\rm ext}_{p,\delta} = w^{\rm ext}_{p,\delta}\circ \Phi$. In view of \eqref{eq:Pdef_o}, we check that $\widetilde{w}^{\rm ext}_{p,\delta} = \widetilde{w^{\rm ext}_{p,\delta}}$ and $\widetilde{w}^{\rm ext}_{p,\delta}$ is stationary.
By \eqref{eq:regest}, we have, for any $K \subset \subset \R^d$,
\begin{equation}
\|\nabla \widetilde{w}^{\rm ext}_{p,\delta}\|_{L^2(\Omega,(L^2(K))^d)} \le C(K).
\end{equation}
That is, $\nabla \widetilde{w}^{\rm ext}_{p,\delta}$ is bounded in $L^2(\Omega,(L^2(K))^d)$. Hence, there exists a subsequence that converges weakly in $L^2(\Omega,(L^2(K))^d)$. The same reasoning applies to the family $\{\nabla \widetilde{w}^-_{p,\delta}\}$ and $\{\widetilde{w}^+_{p,\delta} - \widetilde{w}^-_{p,\delta}\}$ as well.

Recall that $L^2(\Omega)$ is separable; by taking a sequence of $K$'s that exhaust $\R^d$, we obtain a vector field $\widetilde{\xi}_p \in L^2(\Omega, (L^2_{\rm loc}(\R^d))^d)$, $\widetilde{\eta}_p \in L^2(\Omega,(L^2_{\rm loc}(\Rdm))^d)$, $\widetilde{\zeta}_p \in L^2(\Omega, L^2_{\rm loc}(\Gamd))$ and a subsequence $\delta_k \to 0$, such that
\begin{equation}\label{e.subcorr}
\begin{aligned}
&\nabla \widetilde{w}^{\rm ext}_{p,\delta_k} \ \longrightarrow \ \widetilde{\xi}_p &\quad  &\text{weakly in }\; L^2(\Omega,L^2_{\rm loc}(\R^d)^d),\\
&\nabla \widetilde{w}^-_{p,\delta_k} \ \longrightarrow \ \widetilde{\eta}_p &\quad  &\text{weakly in }\; L^2(\Omega,L^2(Y^-_k)^d), \text{ for each } k \in \Z^d,\\
&\widetilde{w}^+_{p,\delta_k} - \widetilde{w}^-_{p,\delta} \ \longrightarrow \ \widetilde{\zeta}_p &\quad  &\text{weakly in }\; L^2(\Omega,L^2(\Gamma_k)) \text{ for each } k \in \Z^d.
\end{aligned}
\end{equation}
For notational simplicity, we denote the sequence $\delta_k$ still by $\delta$. It is easy to check that $\widetilde{\xi}_p, \widetilde{\eta}_p$ and $\widetilde{\zeta}_p$ inherit stationarity from the sequence, and $\widetilde{\xi}_p$, $\widetilde{\eta}_p$ remain potential field in the sense that
\begin{equation*}
\int_\Omega \int_{\R^d} \widetilde{\xi}_p(y,\omega) \cdot \nabla \varphi(y,\omega) \, dy d\Pb(\omega) = 0, \quad \quad \text{for all } \varphi(\cdot,\omega) \in C^\infty_0(\R^d).
\end{equation*}
It follows that there exists $\widetilde{w}^{\rm ext}_{p} \in L^2(\Omega,H^1_{\rm loc}(\R^d))$ such that $\nabla \widetilde{w}^{\rm ext}_p = \widetilde{\xi}_{p}$ (note that the label `ext' here is just a notation and does not mean extension). We note \eqref{eq:auxil_avg} is satisfied by $\widetilde{w}^{\rm ext}_{p,\delta}$, and it is preserved by the limit $\widetilde{w}^{\rm ext}_p$.

Set $\widetilde{w}_p^+$ to be the restriction of $\widetilde{w}^{\rm ext}_p$ in $\Rdp$, and it has a trace on $\Gamd$. Set
\begin{equation}\label{e.wmtrace}
\widetilde{w}^-_p := \widetilde{\zeta}_p - \widetilde{w}^+_p, \quad \text{ on each } \Gamma_k, \ k \in \Z^d,
\end{equation}
where $\widetilde{w}^+_p$ is understood as the trace and the above defines a function in $L^2(\Omega,L^2_{\rm loc}(\Gamd))$. Finally, on each $Y^-_k$, by solving a Dirichlet problem, we find $\widetilde{w}^-_p$ on each $Y^-_k$, such that $\nabla \widetilde{w}^-_p = \nabla \widetilde{\eta}_p$ and $\widetilde{w}^-_p$ has trace defined by \eqref{e.wmtrace} on each $\partial Y^-_k$.

Let $\widetilde{w}_p = (\widetilde{w}^+_p,\widetilde{w}^-_p)$, with $\widetilde{w}^+_p$ and $\widetilde{w}^-_p$ defined above. Set $w_p = \widetilde{w}_p \circ \Phi^{-1}$. In \eqref{eq:weakregaux}, for each fixed $\phi \in \Hcal$ with compact support in $\R^d$, pass to limit through the subsequence found above, use \eqref{e.subcorr}, $\widetilde{w}^+_p - \widetilde{w}^-_p = \widetilde{\zeta}_p$, and the relation of changed variable, we get
\begin{equation}
\begin{aligned}
& \E \int_{\Phi(K\cap \Rdp,\omega)} A(\nabla w_{p}^+ + p)\cdot \nabla \phi^+ dy + \E \int_{\Phi(K\cap \Rdm,\omega)}  A(\nabla w_{p}^- + p)\cdot \nabla \phi^- dy \\
+ & \quad + \E \int_{\Phi(K \cap \Gamd,\omega)} (w^+_{p} - w^-_{p})(\phi^+ -\phi^-) d\sigma(y) = 0.
\end{aligned}
\label{eq:weakcorr}
\end{equation}
Note that the third term in \eqref{eq:weakregaux} goes to zero in the limit because of the last inequality in \eqref{eq:regest} and the fact that $K$ is bounded. This proves that $w_p$ is an annealed solution to the auxiliary problem \eqref{eq:auxil}. In view of Proposition \ref{prop:weak}, $w_p(\cdot,\omega)$ is a weak solution to \eqref{eq:auxil} for a.e. $\omega \in \Omega$.

{\em Sublinearity of $w_p$ at infinity.} In view of the relation between $w^+_p$ and $w^-_p$, it suffices to show that $\widetilde{w}_p^{\rm ext}$ is sublinear at infinity. Owing to Lemma \ref{lem:subli} and \eqref{eq:auxil_avg} established earlier, sublinearity of $w_p$ follows once we show that
\begin{equation}\label{e.gradbdd}
\E \|\nabla \widetilde{w}^{\rm ext}_p\|^s_{L^s(Y)} \le C \quad  \text{ for some } \quad s > d.
\end{equation}
To prove this, we appeal to elliptic regularity theory. We apply the method of difference quotient (e.g. see Evans \cite[Section 6.3]{Evans}) in a finite union of open sets that cover the deformed unit cube $\Phi(Y)$. Since $A \in C^1$, $\partial Y^- \in C^2$ and $\Phi \in C^2$ with bounds that are uniform in $\omega$, we check $\|D \nabla w^+_{p}\|_{L^2(\Phi(Y^+,\omega))} \le C\|\nabla w^+_{p}\|_{L^2((2Y)^+,\omega)}$ with some constant $C$ that is uniform in $\omega$. By Sobolev imbedding, we have $\|\nabla \widetilde{w}^+_p\|^s_{L^s(\Phi(Y^+,\omega))} \le C\|\nabla \widetilde{w}^+_p\|_{L^2(\Phi(2Y^+,\omega)}^{s/2}$. This holds for all $s \in (d,4)$, $d = 2,3$. Note that $s/2 \le 2$ and that the first inequality in \eqref{eq:regest} is preserved by the limit $\nabla w_p$, we verify that \eqref{e.gradbdd} holds and that $w^+_p(\cdot,\omega)$ is sublinear for a.e. $\omega \in \Omega$.

{\em Uniqueness} (of $\nabla w_p)$. Suppose there are two solutions. Let $v_p$ be their difference, then $v_p$ satisfies \eqref{eq:auxil} with $p$ replaced by zero.
Integrate this equation against $v_p$ over the deformed cube $\Phi(Q_N,\omega)$ where $Q_N = (-N,N)^d$. We get
\begin{equation*}
\begin{aligned}
\int_{\Phi(Q_N\cap \Rdp,\omega)} &\nabla v_p^+ \cdot A\nabla v_p^+ dx \ + \ \int_{\Phi(Q_N\cap \Rdm,\omega)}  \nabla v_p^- \cdot A \nabla v_p^- dx \\
 &+\ \int_{\Phi(Q_N\cap \Gamd,\omega)} |v^+_p - v^-_p|^2 d\sigma(x)
\ =\ \int_{\partial \Phi(Q_N,\omega)} v_p^+ \nu_x \cdot A \nabla v_p^+ d\sigma(x).
\end{aligned}
\end{equation*}
The integrals on the left can be written as sum of integrals over $\Phi(Y_k,\omega)$ for $k \in \Z^d$ such that $Y_k \subset Q_N$; the number of such cubes are $(2N+1)^d$. For the integral on the right, we observe that $|v^+_p| = o(N)$ by sublinearity proved above; moreover, it can be written as sum of integrals over parts of $\Phi(\partial Y_k, \omega)$ for $k$ such that $\overline{Y}_k$ intersects $\partial Q_N$; the total number is of order $d(2N+1)^{d-1}$. Divide the above equality by $(2N+1)^d$, and change variable in the integrals, we have
\begin{equation*}
\begin{aligned}
&\frac{1}{(2N+1)^d} \left(\int_{Q_N\cap \Rdp} |\nabla \widetilde{v}^+_p|^2 dx \ +\  \int_{Q_N\cap \Rdm} |\nabla \widetilde{v}^-_p|^2 dx \ +\  \int_{Q_N \cap \Gamd} |\widetilde{v}^+_p - \widetilde{v}^-_p|^2 d\widetilde{\sigma}(x)\right)\\
=\ &\frac{o(N)}{2N+1} \frac{1}{d(2N+1)^{d-1}} \sum_{k \in \mathcal{K}_{\partial Q_N}} \| \nabla \widetilde{v}^{\rm ext}_p \|_{L^2(Y_k)}.
\end{aligned}
\end{equation*}
Send $N$ to infinity and use the ergodic theorem; we conclude that
$$
\E \left(\int_{Y^+} |\nabla \widetilde{v}^+_p|^2 dx \ +\  \int_{Y^-} |\nabla \widetilde{v}^-_p|^2 dx \ +\  \int_{\Gamma_0} |\widetilde{v}^+_p - \widetilde{v}^-_p|^2 d\widetilde{\sigma}(x)\right) = 0,
$$
which implies that $\widetilde{v}_p = v_p = C$, i.e. the two solutions are different by a constant. This completes the proof.
\end{proof}

\section{Proof of the Homogenization Result}
\label{sec:proof}
\subsection{Oscillating test functions}

We start with the construction of oscillating test functions. For any $p \in \R^d$, let $(w^+_p,w^-_p) \in \mathcal{H} = L^2(\Omega,H^1_{\rm loc}(\Phi(\Rdp))\times H^1_{\rm loc}(\Phi(\Rdm)))$ be the solution to the auxiliary problem \eqref{eq:auxil} provided by Theorem \ref{thm:auxil}, and let $w^{\rm ext}_p$ be the corresponding extension of $w^+_p$. Define, for $x \in \R^d$,
\begin{equation}
\label{eq:osctest} 
w^\eps_{1p}(x,\omega) =  x\cdot p + \eps w^{\rm ext}_p \left(\frac{x}{\eps},\omega\right), \quad\quad
w^\eps_{2p}(x,\omega) =  x\cdot p + \eps Qw^-_p
\left(\frac{x}{\eps},\omega\right).
\end{equation}
Here and in the sequel, $Q$ denotes the trivial extension operator
which sets $Qf=0$ outside the spatial support of $f$. By scaling the
auxiliary problem, we find 
\begin{equation*}
\left\{
\begin{aligned}
&-\nabla \cdot(A \nabla w^{\eps}_{1p}) = 0 &\ \text{ in } \eps \Phi(\Rdp) &\quad\text{and}&\quad
-\nabla \cdot(A \nabla w^{\eps}_{2p}) = 0 &\ \text{ in } \eps\Phi(\Rdm),\\
&\frac{\partial w^{\eps}_{1p}}{\partial \nu_{A^\eps}} =  \frac{\partial w^{\eps}_{2p}}{\partial \nu_{A^\eps}}
&\ \text{ on } \eps \Phi(\Gamd) &\quad\text{and}&\quad w^{\eps}_{1p} - w^{\eps}_{2p} = \eps \frac{\partial w^{\eps}_{2p}}{\partial \nu_{A^\eps}} &\ \text{ on } \eps\Phi(\Gamd),
\end{aligned}
\right.
\end{equation*}
in the distributional sense a.e.; that is, for almost every $\omega \in \Omega$, for any test function $\varphi = (\varphi^+,\varphi^-) \in H^1_{\rm loc}(\Phi(\Rdp,\omega)) \times H^1_{\rm loc}(\Phi(\Rdm,\omega)) $ with support $K \subset \subset \R^d$,
\begin{equation}
\begin{aligned}
\int_{K\cap \eps\Phi(\Rdp)} A \nabla w^{\eps}_{1p} \cdot \nabla \varphi^+ dx  &+ \int_{K\cap \eps\Phi(\Rdm)} A \nabla w^{\eps}_{2p} \cdot \nabla \varphi^- dx\\
&+ \frac{1}{\eps}\ \int_{K\cap \eps\Phi(\Gamd)}
(w^{\eps}_{1p} - w^{\eps}_{2p}) (\varphi^+- \varphi^-) d\sigma(x) =
0.
\end{aligned}
\label{eq:wep12weak}
\end{equation}

Let $w^{\eps +}_p$ denote the restriction of $w^{\eps}_{1p}$ on $\Phi(\Rdp,\omega)$ and $w^{\eps - }_p$ the restriction of $w^{\eps}_{2p}$ on $\Phi(\Rdm,\omega)$. We have the following facts about $w^\eps_{jp}$, $j=1,2$, and their gradients.

\begin{lemma} \label{lem:osctest} For each $p \in \R^d$, there exists $\widetilde\Omega \in \F$ with $\Pb(\widetilde\Omega) = 1$, such that for any $\omega \in \widetilde\Omega$, and for any bounded open subset $\mathscr{O} \subset \R^d$, we have
\begin{align}
&w^\eps_{1 p} \rightarrow x\cdot p, \quad &&\text{ uniformly in }
\mathscr{O},
\label{eq:limweext}\\
&w^\eps_{2p} \rightarrow x\cdot p, \quad &&\text{ in }
L^2(\mathscr{O}),
\label{eq:limwem}\\
&Q(A^\eps \nabla w^{\eps \pm}_p) \rightharpoonup \frac{1}{\varrho}\ \E \int_{\Phi(Y^\pm,\omega)} A(y,\omega)
\left(p + \nabla w^\pm_p(x,\omega)\right)dx, &&\text{ in }
(L^2(\mathscr{O}))^d.
\label{eq:limgwe}
\end{align}
\end{lemma}

\begin{proof} For the first result, recall that in step two of the proof of Theorem \ref{thm:auxil}, we proved $w^{\rm ext}_p$ is sublinear a.e. in $\Omega$. In particular, there exists a subset $\Omega_1$ of $\Omega$ with full measure such that, for any $\omega \in \Omega_1$ and for any compact set $\mathscr{O}\subset \R^d$, $d = 2,3$, 
\begin{equation*}
\lim_{\eps \to 0} \ \sup_{x\in \mathscr{O}} \ \left\lvert \eps w^{\rm ext} \left(\frac{x}{\eps}, \omega\right) \right\rvert = 0.
\end{equation*}
This is precisely \eqref{eq:limweext}.

For the second convergence result, we first write
$$
w^\eps_{2p} - x\cdot p = \eps \left(w^-_p \left(\frac{x}{\eps}\right) - w^{\rm
ext}_p \left(\frac{x}{\eps}\right) \right) \chi_{\eps\Phi(\Rdm)} + \eps
w^{\rm ext}_p\left(\frac{x}{\eps}\right) \chi_{\eps\Phi(\Rdm)}.
$$
In view of \eqref{eq:limweext}, the second item on the right converges uniformly in $\mathscr{O}$ to zero. Let $J_\eps$ denotes the $L^2(\mathscr{O})$ norm of the first item, then it suffices to show that $J_\eps$ converges to zero. Let $\mathcal{I}_\eps(\mathscr{O}) = \{k \in \Z^d ~|~ \eps \Phi(Y_k) \cap \mathscr{O} \ne \emptyset\}$. Then $\mathscr{O} \subset \cup_{k \in \mathcal{I}_\eps} \eps \Phi(Y_k)$ and $|\mathcal{I}_\eps| \lesssim C(\mathscr{O}) \eps^{-d}$, where the constant $C$ can be chosen to be independent of $\eps$ and $\omega$. We have, in view of (A2), (A3) and \eqref{eq:Psibdd},
\begin{equation*}
\begin{aligned}
J^2_\eps &\le \sum_{n \in \mathcal{I}_\eps} \int_{\eps\Phi(Y^-_n)} \eps^2
 \left|w^{\rm ext}_p \left(\frac{x}{\eps}\right) - w^-_p \left(\frac{x}{\eps}\right) \right|^2 dx = \eps^{d+2} \sum_{n
  \in \mathcal{I}_\eps} \int_{\Phi(Y^-_n)} \left|w^{\rm ext}_p(x) - w^-_p(x)\right|^2 dx\\
&\le C\eps^{d+2} \sum_{n \in \mathcal{I}_\eps} \int_{Y^-_n}
\left|\widetilde{w}^{\rm ext}_p(y) -
\widetilde{w}^-_p(y)\right|^2 dy.
\end{aligned}
\end{equation*}
Using the estimate that $\|f\|_{L^2(Y^-)} \le C\|\nabla f\|_{L^2(Y^-)} + C\|f\|_{L^2(\partial Y^-)}$, we have
\begin{equation*}
J^2_\eps \le C\eps^2 \left[\frac{1}{|\mathcal{I}_\eps|} \sum_{n \in
\mathcal{I}_\eps} \left(\int_{\Gamma_n}
\left|\widetilde{w}^+_p(y) - \widetilde{w}^-_p(y)\right|^2
d\sigma(y) + \int_{Y^-_n} \left| \nabla \widetilde{w}^{\rm
ext}_p(y) - \nabla \widetilde{w}^-_p(y)\right|^2
dy \right)\right].
\end{equation*}
Note that the integrands above are stationary and the item inside
the bracket is ready for applying ergodic theorem, which yields: there exists $\Omega_2 \in \F$, $\Pb(\Omega_2) = 1$, and for all $\omega \in \Omega_2$, the term in the bracket above converges to
\begin{equation*}
\E \int_{\Gamma_0} |\widetilde{w}^+_p - \widetilde{w}^-_p|^2(y)
d\sigma(y) + \E \int_{Y^-} |\nabla \widetilde{w}^{\rm ext}_p - \nabla
\widetilde{w}^-_p|^2 dy.
\end{equation*}
Since the bounds \eqref{eq:regest} are preserved by the limits $(w^+_p,w^-_p)$, we find $\lim_{\eps \to 0} J_\eps = 0$, from which \eqref{eq:limwem} follows.

For the third convergence result, we observe that
\begin{equation*}
Q(A^\eps \nabla w^{\eps \pm}_p) (x,\omega)= \left(\left(\chi_{\R_d^\pm} \widetilde{A} \left\{p + \left((D\Phi)^{-1}\right)^t \nabla
\widetilde{w}^\pm_p\right\} \right) \circ \Phi^{-1}\right) \left(\frac{x}{\eps},\omega\right).
\end{equation*}
We note that the functions $\chi_{\R_d^\pm} \widetilde{A} \left\{p + \left((D\Phi)^{-1}\right)^t \nabla
\widetilde{w}^\pm_p\right\}$ are stationary and belong to $L^2(\Omega,L^2_{\rm loc}(\R^d))$. In view of Lemma \ref{lem:gstat}, we conclude that there exists $\Omega_3 \in \F$ with full measure such that \eqref{eq:limgwe} holds for all $\omega \in \Omega_3$.

Set $\widetilde{\Omega} = \Omega_1 \cap \Omega_2 \cap \Omega_3$. Then we complete the proof of the lemma.
\end{proof}

\subsubsection{Proof of the homogenization theorem}

We follow the standard method of oscillating test functions. Let $\Omega_*$ be the set with full measure such that \eqref{eq:wep12weak} and the conclusion in Lemma \ref{lem:osctest} hold for the unit vectors $p \in \{e_k ~|~ k = 1,\cdots, d\}$. The strategy of proof is as follows: In the first step, we recall uniform in $\omega$ and $\eps$ energy estimates for the solution $\ueps$ to the problem \eqref{eq:aniso} and, for each fixed $\omega \in \Omega_*$, we extract a subsequence, along which $u^{\rm ext}_\eps$ converges weakly in $H^1(D)$ to some $u_0$ and the trivially extended flux $Q(A^\eps\nabla \uepsp)$ and $Q(A^\eps\nabla \uepsm)$ converge weakly in $(L^2(D))^d$; a priori, both the subsequence and the limit depend on $\omega$. In the second step, we apply the oscillating test function $(\varphi w^\eps_{1p}, \varphi w^\eps_{2p})$ to \eqref{eq:weakaux_pw} and use the test function $(\varphi \uepsp,\varphi\uepsm)$ in  \eqref{eq:wep12weak}. Passing to limits, we obtain the equation with proper boundary conditions satisfied by $u_0$. It turns out that $u_0$ solves a deterministic problem which has unique solution. As a result, the whole sequence $\ueps$ converges to $u_0$, establishing the desired result. As a by-product, we also prove that the trivial extension $Q\uepsm$ converges weakly in $L^2(D)$ to $\theta u_0$ for some constant $\theta$ strictly less than one.

\bigskip

\begin{proof}[Proof of Theorem \ref{thm:homog}] {\itshape Step 1. A converging subsequence.} Fix any $\omega$ in the set $\Omega_*$ defined above. For notational simplicity, we omit the dependence of functions on $\omega$ below. From \eqref{eq:BddExt} and \eqref{eq:GradientEst} we find, for some $C$ independent of $\eps$ and $\omega$,
\begin{equation*}
\|\ueps^{\rm ext}\|_{H^1(D)} + \|Q(A^\eps\nabla \ueps^+)\|_{(L^2(D))^d}
+ \|Q(A^\eps\nabla \ueps^-)\|_{(L^2(D))^d} \le C.
\end{equation*}
As a result, there exist $u_0 \in H^1(D)$, vector fields $\xi_1, \xi_2 \in (L^2(D))^d$, and a subsequence of $\ueps$ still indexed by $\eps$, such that
\begin{equation}
\begin{aligned}
&\ueps^{\rm ext} \rightharpoonup u_0 \text{ weakly in } H^1(D), \quad &&Q(A^\eps\nabla \ueps^+) \rightharpoonup \xi_1 \text{ weakly in }
(L^2(D))^d;\\
&\ueps^{\rm ext} \rightarrow u_0 \text{ strongly in }
L^2(D), \quad &&Q(A^\eps\nabla \ueps^-) \rightharpoonup \xi_2 \text{
weakly in } (L^2(D))^d.
\end{aligned}
\label{eq:usubseq}
\end{equation}
In the proof of Proposition \ref{cor:extlim}, we also proved that
\begin{equation}
\ueps^{\rm ext} \chi_{\eps}^- - Qu^-_\eps \rightarrow 0 \text{
strongly in } L^2(D), \label{eq:usubseq1}
\end{equation}
so in fact $\ueps = \ueps^{\rm ext} + (-\ueps^{\rm ext} \chi_{\eps}^- + Qu^-_\eps)$ converges to $u_0$ strongly in $L^2(D)$. We note that, at this stage, the limiting functions $u_0, \xi_1, \xi_2$ and the subsequence all depend on the fixed $\omega$. At the end of step 2, however, it will be evident that the whole sequence converges and the limits are deterministic.

\smallskip

{\itshape Step 2: Equation for $u_0$.} Fix an arbitrary test function $\varphi \in
C^\infty_0(D)$ with support $K \subset \subset D$. Take $(\varphi \chi_{\eps}^+,
\varphi\chi_{\eps}^-)$ as the test function in \eqref{eq:rpdeweak}.
Then the interface term disappears and
\begin{equation*}
\int_{D} Q(A^\eps \nabla u^+_\eps) \cdot {\nabla \varphi} dx +
\int_{D} Q(A^\eps \nabla u^-_\eps) \cdot {\nabla \varphi} dx =
\int_D f \varphi dx.
\end{equation*}
Passing to the limit $\eps \to 0$ along the chosen subsequence, one
finds
\begin{equation}
\int_D (\xi_1 + \xi_2) \cdot {\nabla \varphi} dx =
\int_D f\varphi dx. \label{eq:xilim}
\end{equation}
In other words, $-\mathrm{div}\,(\xi_1 + \xi_2) = f$ in the distributional sense.

Next, recall the definition of $\Depsm, \Gameps, K_\eps$ and $E_\eps$ in \eqref{eq:Depsdef} and \eqref{eq:KEdef}. For $\eps$ sufficiently small, the function $\varphi$ is compactly supported in $E_\eps$. In particular, we have $\eps\Phi(\Gamd) \cap K = \Gameps \cap K$ where $K$ is the support of $\varphi$. Let $\{e_k\,:\, k = 1,\cdots, d\}$ denote the standard basis for $\R^d$. Let $w^\eps_{1e_k}$ and $w^\eps_{2e_k}$ be as defined in \eqref{eq:osctest}. Take $(\varphi {\uepsp}, \varphi {\uepsm})$ as the test function in pointwise version of \eqref{eq:wep12weak}; then for each $k = 1,2,\cdots,d$, we get
\begin{equation*}
\begin{aligned}
\int_D Q(A^\eps \nabla w^{\eps+}_{e_k}) \cdot \nabla({\varphi}
\uepsp) dx \; + \; &\int_D Q(A^\eps \nabla w^{\eps-}_{e_k}) \cdot
\nabla({\varphi} \uepsm) dx\\
\; + \; &\frac{1}{\eps}
\int_{\Gameps} (w^\eps_{1e_k} - w^\eps_{2e_k})
{\varphi}(\uepsp - \uepsm) ds \; = \; 0.
\end{aligned}
\end{equation*}
Similarly, take $(\varphi {w^\eps_{1e_k}}, \varphi {w^\eps_{2e_k}})$ as the test function in \eqref{eq:rpdeweak}; we get
\begin{equation*}
\begin{aligned}
\int_D Q(A^\eps \nabla u^+_\eps) \cdot \nabla({\varphi}
w^\eps_{1e_k}) \; + \; &\int_D Q(A^\eps \nabla u^-_\eps) \cdot
\nabla({\varphi} w^\eps_{2e_k})\\
\; + \; \frac{1}{\eps}&
\int_{\Gameps} (\uepsp - \uepsm){\varphi}(w^\eps_{1e_k}
- w^\eps_{2e_k}) \; = \; \int_D f\varphi (\chi^+_\eps w^\eps_{1e_k} + \chi^-_\eps w^\eps_{2e_k}).
\end{aligned}
\end{equation*}
Subtract the two equalities above. Then the interface
terms cancel out; in view of the definitions \eqref{eq:osctest}, the terms in which the derivative does not land on $\varphi$ also cancel out. We obtain
\begin{equation*}
\begin{aligned}
\int_D &\left[ Q(A^\eps \nabla w^{\eps+}_{e_k}) + Q(A^\eps \nabla w^{\eps-}_{e_k})\right] \cdot {\nabla \varphi} \ueps^{\rm ext} dx 
\ +  \int_D Q(A^\eps \nabla w^{\eps-}_{e_k}) \cdot {\nabla \varphi} (Q\ueps^- - \ueps^{\rm ext}\chi_{\eps}^-) dx\\
- &\int_D \left[ w^\eps_{1e_k} Q(A^\eps \nabla u^{+}_{\eps}) + w^\eps_{2e_k} Q(A^\eps \nabla u^{-}_{\eps})\right] \cdot {\nabla \varphi} dx  \ + \int_D f\varphi (\chi^+_\eps w^\eps_{1e_k} + \chi^-_\eps w^\eps_{2e_k}) \ =\ 0
\end{aligned}
\end{equation*}
In view of \eqref{eq:limgwe}, \eqref{eq:limweext},
\eqref{eq:limwem}, \eqref{eq:usubseq} and \eqref{eq:usubseq1}, we
observe that each of integrand in the first three integrals is a product of a strong
converging term with a weak converging one. Similarly, we show in step 3 below that $\chi^-_\eps$ (respectively $\chi^+_\eps$) converges weakly in $L^2$ to a constant $\theta \in (0,1)$ (and $1-\theta$); so the integrand of the last integral also converges. Pass to the limit and let $\eta_{1e_k}$ and $\eta_{2e_k}$ be the limit of $Q(A^\eps \nabla w^{\eps +}_{e_k})$ and $Q(A^\eps \nabla w^{\eps -}_{e_k})$ as in \eqref{eq:limgwe}; we have
\begin{equation}
\int_D (\eta_{1e_k} + \eta_{2e_k}) u_0 \cdot { \nabla
\varphi} dx - \int_D  x_k (\xi_1 + \xi_2) \cdot
{\nabla \varphi} dx + \int_D f\varphi x_k dx = 0.
\label{eq:xieta}
\end{equation}
For the first term, using the definitions of $\eta_{1e_k}$, $\eta_{2e_k}$ and $A^0$ in \eqref{eq:Abardef}, we have
\begin{equation*}
\begin{aligned}
(\eta_{1e_k} + \eta_{2e_k}) \cdot \nabla\varphi &= \sum_{\ell =1}^d \frac{1}{\varrho} \E\left( \int_{\Phi(Y^+)} a_{\ell k}(e_k + \nabla w^+_{e_k}) +  \int_{\Phi(Y^-)} a_{\ell k}(e_k + \nabla w^-_{e_k})\right)\frac{\partial \varphi}{\partial x_{\ell}}\\
&= \sum_{\ell = 1}^d a^0_{k\ell} \frac{\partial \varphi}{\partial x_{\ell}} = e_k \cdot A^0 \nabla \varphi.
\end{aligned}
\end{equation*}
For the second and third items in \eqref{eq:xieta}, we apply \eqref{eq:xilim} with the function $\varphi x_k$ to combine the terms. This yields
\begin{equation*}
\int_D (u_0 e_k \cdot A^0) \cdot \nabla \varphi dx = - \int_D e_k \cdot (\xi_1 + \xi_2) \varphi dx,
\end{equation*}
which yields that $(\xi_1 + \xi_2)\cdot e_k =  \mathrm{div} (u_0 e_k \cdot A^0) = e_k \cdot A^0 \nabla u_0$. Therefore, the vector field $\xi_1 + \xi_2$ coincides with $A^0 \nabla u_0$. In light of \eqref{eq:xilim}, we have
\begin{equation}\label{eq:hom_eq}
- \nabla \cdot (A^0 \nabla u_0) = f \quad \quad \text{ in } D
\end{equation}
in the distributional sense.

Next we consider the boundary condition that is satisfied by $u_0$. Recall that $u^{\rm ext}_\eps \rightharpoonup u_0$ weakly in $H^1(D)$, $u^{\rm ext}_\eps \vert_{\partial D} = 0$ for all $\eps$, and that the trace operator from $H^1(D)$ to $L^2(\partial D)$ is continuous with respect to the weak topology. Therefore,
\begin{equation}\label{eq:hom_bc}
u_0 = 0 \quad\quad \text{ on } \partial D.
\end{equation}
Combining \eqref{eq:hom_eq} and \eqref{eq:hom_bc} together, we conclude that $u_0 \in H^1_0(D)$ and it solves the deterministic problem \eqref{eq:homog}. Finally, provided that $A^0$ is uniformly elliptic which we prove in Section \ref{sec:discussion}, it is obvious that \eqref{eq:homog} has a unique solution in $H^1_0(D)$. As a result, the whole sequence $\{u^{\rm ext}_\eps\}$ converge strongly in $L^2(D)$ and weakly in $H^1_0(D)$ to $u_0$, the unique solution to \eqref{eq:homog}. Items (i), (ii) and (iv) of Theorem \ref{thm:homog} are proved.

\smallskip

{\itshape Step 3: Convergence of $Q\uepsm$.} We can write
$Q\uepsm$ as $\ueps^{\rm ext} \chi_{\eps}^- + (Q\uepsm -
\ueps^{\rm ext} \chi_{\eps}^-)$ where $\chi_{\eps}^-$ is the indicator function of $D_\eps^-$. Due to \eqref{eq:usubseq1} and
the fact that $\ueps^{\rm ext}$ converges strongly to $u_0$, we
only need to verify that $\chi_{\eps}^-$ converges weakly in $L^2(D)$ to
$\theta$. For this purpose, fix an arbitrary open set $K \subset \subset D$. Then for sufficiently small
$\eps$, $K$ lies in $E_\eps$ defined in
\eqref{eq:KEdef}. We have
\begin{equation*}
\int_K \chi_\eps^- dx = \int_{K\cap \eps \Phi(\Rdm)} dx
= \int_K \chi_{\eps \Phi(\Rdm)}(x) dx.
\end{equation*}
We observe that $x \in \eps\Phi(\Rdm)$ if and only if $\Phi^{-1}\left(\frac{x}{\eps}\right) \in \Rdm$, which yields
\begin{equation*}
\chi_{\eps \Phi(\Rdm,\omega)}(x) = \chi_{\Rdm} \left(\Phi^{-1}\left(\frac{x}{\eps}, \omega\right)\right),
\end{equation*}
and apparently $\chi_{\Rdm}(z)$ is periodic and hence stationary, and it is uniformly bounded. By Lemma \ref{lem:gstat}, the above function converges in $L^\infty$ weak-$*$ topology. More precisely, in view of the definitions in \eqref{eq:rho_theta}, we have for almost all $\omega \in \Omega$,
\begin{equation}
\chi_{\eps \Phi(\Rdm,\omega)} \xrightarrow[\eps \to 0]{L^\infty \text{ weak-}*} \frac{1}{\varrho} \E \int_{\Phi(Y^-)} dx  = \theta.
\end{equation}
Upon redefining $\Omega_*$ by intersection, we will assume that the above is valid for all $\omega \in \Omega_*$ that was chosen at the beginning of step one. As a result, we have that
\begin{equation*}
\lim_{\eps \to 0} \int_{K} \chi^-_\eps(x,\omega) dx =  \int_K \theta\ dx = \theta |K|, \quad\quad \forall K \subset \subset D.
\end{equation*}
By invoking the density of simple functions in $L^2(D)$, we conclude that $\chi^-_\eps$ converges weakly to $\theta$. Hence $Q\uepsm$ converges weakly in $L^2(D)$ to $\theta u_0$. This completes the proof of Theorem \ref{thm:homog}.
\end{proof}

\section{Further Discussions}
\label{sec:discussion}

We first show that the homogenized coefficient $A^0$ defined by \eqref{eq:Abardef} is uniformly elliptic. $A^0$ is clearly bounded from above, so we concentrate on the coercivity of $A^0$. For any vector $\xi \in \R^d$, by the definition of $A^0$ and the linearity of $p \mapsto w_p$, we have
\begin{equation*}
a^0_{ij}\xi_i \xi_j = \frac{1}{\varrho}\ \E\left( \int_{\Phi(Y^+,\omega)} \xi \cdot A(\xi+ \nabla w^+_{\xi}) dx + \int_{\Phi(Y^-,\omega)} \xi \cdot A (\xi + \nabla w^-_{\xi}) dx\right).
\end{equation*}
Take $p = \xi$ in the auxiliary problem \eqref{eq:auxil} and take $w_\xi$ as the test function. By an argument that is similar to the uniqueness step in the proof of Theorem \ref{thm:homog}, we verify that
\begin{equation*}
\E\left( \int_{\Phi(Y^+)} \nabla w^+_{\xi} \cdot A(\xi+ \nabla w^+_{\xi}) + \int_{\Phi(Y^-)} \nabla w^-_{\xi} \cdot A (\xi + \nabla w^-_{\xi}) + \int_{\Phi(\Gamma_0)} |w^+_\xi - w^-_\xi|^2 \right) = 0,
\end{equation*}
which yields
\begin{equation*}
\begin{aligned}
a^0_{ij}\xi_i \xi_j = & \frac{1}{\varrho} \ \E \int_{\Phi(\Gamma_0,\omega)} | w^+_{\xi} - w^-_{\xi}|^2 d\sigma(x) + \frac{1}{\varrho}\ \E \int_{\Phi(Y^+,\omega)} (\xi + \nabla w^+_{\xi}) \cdot A(\xi+ \nabla w^+_{\xi}) dx\\
&\ +   \frac{1}{\varrho} \E \int_{\Phi(Y^-,\omega)} (\xi + \nabla w^-_{\xi}) \cdot A (\xi + \nabla w^-_{\xi}) dx \, \ge 0.
\end{aligned}
\end{equation*}
This shows that $a^0_{ij}$ is positive semidefinite. Further, the inequality above becomes equality if and only if $w_\xi$ has no jump across $\Gamma_0$ and $\nabla w_{\xi} = -\xi$. Then $\nabla \widetilde{w}_{\xi} = -(D\Phi)^t \xi$. Recall that the integral of $\nabla \widetilde{w}_{\xi}$ over the unit cube has mean zero and $D\Phi$ is non-degenerate; this forces $\xi$ to be zero. It follows that $A^0$ is uniformly elliptic.

Our analysis applies to several variations of the problem \eqref{eq:aniso}. First, like in \cite{Monsur}, instead of assuming that the materials across the interfaces are the same, we may consider two different materials. This amounts to using two different elliptic coefficients $A^1$ and $A^2$ in \eqref{eq:perio}. Secondly, similar to \cite{AGGJS13}, rather than considering the Dirichlet boundary condition at $\partial D$, we may treat also Neumann or Robin type boundary conditions, and we may consider non-homogeneous boundary conditions as well. Finally, if the underlying application is not in biology, the modification near the boundary $\partial D$ at the end of Section \ref{sec:diffe} is not necessary, since we see already that in the homogenization proof we only need to take test functions that are compactly supported in $D$. These claims can be rigorously justified by examining our analysis and making slight modifications.

We conclude this paper by some interesting questions beyond homogenization which are out of the scope of this paper. The current article concerns only the homogenization of \eqref{eq:aniso}, and it is natural to ask about the convergence rate. Such quantitative estimates in stochastic homogenization is much more difficult, but there have been important progresses in several situations, e.g. \cite{CSW05,GO11,GNO13,ACS14}. Once the convergence rate is clear, one may investigate further the detailed structures of the mean of the error and the distributions of the random error, and those in numerical homogenization schemes; see e.g. \cite{FOP82,Bal08,BGMP08,BGGJ12}; see also \cite{BJ11,BJ14}.

\medskip

{\bf Acknowledgement.} The author is grateful to Habib Ammari and Josselin Garnier for suggesting the problem, for their encouragement and for helpful discussions on stochastic homogenization. He would also like to acknowledge stimulating discussions with Christophe Prange on the second step in the proof of Lemma \ref{lem:regul}. We thank the referee for useful suggestions that help to improve the presentation of the paper.
\medskip

\appendix

\section{Appendix}

\subsection{Surface integrals under diffeomorphisms}

We record a geometric fact which concerns surface integrals under a diffeomorphism of the ambient space. 

\begin{proposition}\label{surface-cov}
Let $\Phi: \R^n \to \R^n$ be an orientation preserving diffeomorphism and its inverse is denoted by $\Psi$. Let $\widetilde{S}$ be an $n-1$ dimensional smooth surface given by $\widetilde{S} = g^{-1}\{0\}$ for some smooth function $\widetilde{g} : \R^n \to \R$. Let $S = \Phi(\widetilde{S})$ be the image of $S$ under $\Phi$. Denote by $\nu_y$ and $\widetilde{\nu}_x$ the outward unit normal vector of $S$ at $y$ and that of $\widetilde{S}$ at $x$. We have
\begin{equation}\label{eq:geo-f1}
\widetilde{\nu}_x = \frac{\|\nabla g(\Phi(x))\|}{\|\left(D\Phi(x)\right)^t\nabla g(\Phi(x))\|} \left(D\Phi(x)\right)^{t} \nu_y,
\end{equation}
where $(D\Phi)^t$ is the transpose of the Jacobian matrix.

Let $d\sigma(y)$ and $d\widetilde{\sigma}(x)$, where $y = \Phi(x)$, denote the surface measures on $S$ and $\widetilde{S}$. Then for any integrable function $f$ on $S$, with $\widetilde{f}$ denoting $f\circ \Phi$, we have
\begin{equation}\label{eq:geo-f2}
\int_{S} f(y) d\sigma(y) = \int_{\widetilde{S}} \widetilde{f}(x) \frac{\|\left(D\Phi(x)^{-1}\right)^t \nabla_x \widetilde{g}(x)\| \det(D\Phi)(x)}{\|\nabla_x \widetilde{g}(x)\|} d\sigma(x).
\end{equation}
\end{proposition}

\begin{proof}
The first equality is a calculus fact and it follows from the relations
\begin{equation*}
\nu_y = \frac{\nabla_y g(y)}{\|\nabla_y g(y)\|}, \quad \nu_x = \frac{\nabla_x \widetilde{g}(x)}{\|\nabla_x \widetilde{g}(x)\|},\quad \widetilde{g}(x) = g(\Phi(x)),
\end{equation*}
and an application of the chain rule. We verify the second equality from a geometric point of view. For this purpose, set $M = \R^n$ and $N = \Phi(\R^n)$ so that $\Phi: M \to N$ is a diffeomorphsim between two manifolds. Recall that the volume form $\gamma_{n-1}$ on the surface $S$, which is a submanifold of $N$, is $\iota_{\nu_y} \gamma_n$, i.e. the interior product of the vector $\nu_y$ with the volume form $\gamma_n$ of $N$. Similarly, the volume form $\widetilde{\gamma}_{n-1}$ of $\widetilde{S}$ is $\iota_{\nu_x} \widetilde{\gamma}_n$. By the change of variable formula, we have
\begin{equation*}
\int_S f(y) d\sigma(y) = \int_S f \gamma_{n-1} = \int_{\widetilde{S}} \widetilde{f} \Phi^* \gamma_{n-1},
\end{equation*}
where $\Phi^* \gamma_{n-1}$ is the pull-back of the volume form $\gamma_{n-1}$.

Next we aim to find the relation between $\Phi^* \gamma_{n-1}$ to $\widetilde{\gamma}_{n-1}$. Let $h$ be the function in front of $\left(D\Phi(x)\right)^{t} \nu_y$ in \eqref{eq:geo-f1}. We calculate
\begin{equation*}
\begin{aligned}
\Phi^* \gamma_{n-1} &= \Phi^*\left(\iota_{\nu_y} \gamma_n\right) = \iota_{(D\Phi)^{-1}\nu_y} \Phi^* \gamma_n = \det(D\Phi) \left(\iota_{(D\Phi)^{-1}\nu_y} \widetilde{\gamma}_{n}\right) \\
&= \frac{\det(D\Phi)}{h} \left(\iota_{(D\Psi) \left(D\Psi\right)^t\nu_x} \widetilde{\gamma}_{n}\right) = \frac{\det(D\Phi)}{h} \frac{\widetilde{\gamma}_n ((D\Psi) \left(D\Psi\right)^t\nu_x, X_2,\cdots,X_n)}{\widetilde{\gamma}_n(\nu_x,X_2,\cdots,X_n)}  \left(\iota_{\nu_x} \widetilde{\gamma}_n \right) \\
&= \frac{\det(D\Phi)}{h} \|(D\Psi)^t \nu_x\|^2 \widetilde{\gamma}_{n-1} = h \det(D\Phi) \widetilde{\gamma}_{n-1}.
\end{aligned}
\end{equation*}
In the seond line above, $(\nu_x, X_2,\cdots,X_n)$ is chosen to be an orthonormal frame at $x \in \widetilde{S}$. The second equality of the proposition follows from the above identity.
\end{proof}

\subsection{Proof of Proposition \ref{prop:weak}}
\label{sec:equiv_weak}

In this section, we prove Proposition \ref{prop:weak}. We recall that by assumption, the $\sigma$-algebra $\F$ is countably generated so that $L^2(\Omega)$ is separable.

Proposition \ref{prop:weak} says that for almost all realization $\omega \in \Omega$, the weak solution to the auxiliary problem defined in \eqref{eq:weakaux}, where both variables of $\Omega \times \R^d$ are integrated in the weak formulation, is also a weak solution in the usual sense, where only the spatial variable is integrated. Recall that the space of test functions in the usual weak formulation \eqref{eq:weakaux_pw} is the composition of
\begin{equation*}
C^\infty_c(\Rdp \times \Rdm) := \{\widetilde{\phi} = (\widetilde{\phi}^+, \widetilde{\phi}^-) \,|\, \widetilde{\phi}^+ \in C^\infty(\Rdp), \, \widetilde{\phi}^- \in C^\infty(\Rdm), \, \mathrm{supp}\, \widetilde{\phi} \subset \subset \R^d \}.
\end{equation*}
with $\Phi^{-1}$. Due to the compact support, the space above is separable, and we can choose a countable dense subset of it denoted by $\{\widetilde{\phi}_k\}_{k=1}^\infty$. For any $\omega \in \Omega$, set $\phi_k(\cdot,\omega) = \widetilde{\phi}_k \circ \Phi^{-1}(\cdot, \omega)$.
\smallskip

\begin{proof}[Proof of Proposition \ref{prop:weak}] Let $w_p \in \Hcal$ be a weak solution to \eqref{eq:auxil} in the sense of \eqref{eq:weakaux}. Fix a $\widetilde{\phi}_k$ and let $K$ denote its support. Then for any $\psi \in L^2(\Omega)$, the function $\psi(\omega) \phi_k(x,\omega) \in \Hcal$, and we have
\begin{equation*}
\begin{aligned}
 \int_\Omega \psi(\omega) \left(\int_{\Phi(K\cap \Rdp,\omega)} A(\nabla w_p^+ + p)\cdot \nabla \phi_k^+ dy + \int_{\Phi(K\cap \Rdm,\omega)}  A(\nabla w_p^- + p)\cdot \nabla \phi_k^- dy\right.\\
+  \quad \left.\int_{\Phi(K\cap \Gamd,\omega)} (w^+ - w^-)(\phi_k^+ -\phi_k^-) d\sigma(y) \right)d{\Pb(\omega)}= 0.
\end{aligned}
\end{equation*}
Since $\psi$ is arbitrary, the sum of the inner integrals must be zero almost surely. In other words, there exists a measurable subset $\Omega_k \subset \Omega$ with $\Pb(\Omega_k) = 1$ and for all $\omega \in \Omega_k$, we have
\begin{equation}
\begin{aligned}
 \int_{\Phi(K\cap  \Rdp,\omega)} A(\nabla w_p^+ + p)\cdot \nabla \phi_k^+ dy + \int_{\Phi(K\cap \Rdm,\omega)}  A(\nabla w_p^- + p)\cdot \nabla \phi_k^- dy\\
+  \quad \int_{\Phi(K \cap \Gamd,\omega)} (w^+ - w^-)(\phi_k^+ -\phi_k^-) d\sigma(y) = 0.
\end{aligned}
\label{eq:weak_pw}
\end{equation}
Set $\Omega_0 = \bigcap_{k=1}^\infty \Omega_k$. Then $\Pb(\Omega_0) = 1$ and \eqref{eq:weak_pw} holds for all $\omega \in \Omega_0$ and all $k=1,2,\cdots$. The general case  \eqref{eq:weakaux_pw} follows by the standard density argument. This completes the proof of Proposition \ref{prop:weak}.
\end{proof}
\subsection{Backward induction}

For the sake of completeness, we include here a proof of the backward induction, which was used in step two in the proof of Lemma \ref{lem:regul}.
\begin{lemma}[Backward induction]\label{lem:bw} Let $E_1\le E_2\le \cdots \le E_n$ be an increasing sequence of $n$ nonnegative real numbers. Suppose $E_n \le Cn^d$ and
\begin{equation}
E_k \le C_1(E_{k+1}-E_k + (k+1)^d), \quad  k=1,2,\cdots,n-1.
\label{eq:bw:sv}
\end{equation}
Then there exists a constant $C' = C'(C,C_1,d)$ such that
\begin{equation}
E_k \le C'k^d, \quad k=1,2,\cdots,n.
\label{eq:bw}
\end{equation}
\end{lemma}
\begin{proof} First, it is easy to see that the conclusion of the lemma follows if we could prove that there exist $C_2>0,C_3>0$ such that
\begin{equation}
E_k \le C_2(k^d + C_3), \quad k = 1,2,\cdots,n.
\label{eq:bw:key}
\end{equation}
Indeed, we can choose $C' = C_2(C_3+1)$ in \eqref{eq:bw}. So, we focus on the proof of \eqref{eq:bw:key}. Let us choose $C_2 = \beta C_1$ for some $\beta>1$ so that $C_2\ge C$. Then for any $C_3>0$, the inequality in \eqref{eq:bw:key} holds for $k=n$. Suppose that this inequality holds for $j=n,n-1,\cdots,k+1$ but
\begin{equation}
E_k > C_2(k^d + C_3) = C_1\beta(k^d + C_3).
\label{eq:bw:Ekbig}
\end{equation}
Then we will have
\begin{equation*}
E_{k+1}-E_k < C_2((k+1)^d + C_3) - C_2(k^d + C_3) = C_2((k+1)^d - k^d).
\end{equation*}
Substitute this relation into \eqref{eq:bw:sv}; we get
\begin{equation*}
E_k \le C_1((C_2+1)(k+1)^d-C_2 k^d).
\end{equation*}
Comparing this inequality with \eqref{eq:bw:Ekbig}. We see that
\begin{equation}
0>C_1 \beta \left[C_3 + k^d\left(1+C_1 - \left(C_1 + \frac{1}{\beta}\right)\left(1+\frac{1}{k}\right)^d\right)\right]
\label{eq:bw:contra}
\end{equation}
Given $C_1>0$ and $\beta>1$. There exists a $k_0(C_1,\beta)$ such that
\begin{equation*}
\left(C_1 + \frac{1}{\beta}\right)\left(1+\frac{1}{k}\right)^d - (C_1 + 1) \begin{cases} 
\le 0 &\text{if}\quad k > k_0,\\
>0 &\text{if} \quad k \le k_0.
\end{cases}
\end{equation*}
Therefore, if we take
$$
C_3 = \max_{k\le k_0} \ k^d \left(\left(C_1 + \frac{1}{\beta}\right)\left(1+\frac{1}{k}\right)^d - (C_1 + 1)\right),
$$
then \eqref{eq:bw:contra} cannot happen. This means that \eqref{eq:bw:key} holds with $C_2=\beta C_1$ and $C_3$ given above. This completes the proof of the backward induction.
\end{proof}

\end{document}